\documentclass[a4paper]{amsart}
\usepackage{amsfonts, amsmath, wasysym}
\usepackage{verbatim}

\usepackage{amssymb,amsthm,
paralist
}

\usepackage{amssymb,amsmath}
\usepackage{enumerate}

\usepackage[latin1]{inputenc}
\usepackage{color}
\usepackage[pdftex]{graphicx}

\usepackage[final]{optional}
\usepackage[british,american]{babel}

\usepackage{mathrsfs}

\usepackage{amsfonts, amsmath, wasysym}

\usepackage{amssymb,amsthm,
paralist
}

\usepackage{
latexsym,
nicefrac,
}

\usepackage{url}

\usepackage{bibgerm}
\usepackage[normalem]{ulem}

\DeclareMathOperator{\dist}{dist}

\DeclareMathOperator{\supp}{supp}
\DeclareMathOperator{\graph}{graph}

\DeclareMathOperator{\divergenz}{div}

\def\ol#1{\overline{#1}}

\def\R{{\mathbb{R}}}
\def\N{{\mathbb{N}}}

\def\O{{\mathcal{O}}}

\def\P{{\mathcal{P}}}

\def\theta{{\vartheta}}
\def\phi{{\varphi}}
\def\epsilon{{\varepsilon}}
\def\dt{{\frac{d}{dt}}}
\def\tdt{{\tfrac{d}{dt}}}
\def\fracd#1#2{{\frac{\displaystyle #1}{\displaystyle #2}}}

\def\tfracp#1#2{{\tfrac{\partial #1}{\partial #2}}}
\def\umbruch{{\displaybreak[1]}}
\long\def\neueZeile{{\rule{0mm}{1mm}\\[-3.25ex]\rule{0mm}{1mm}}}
\long\def\neueZeilealt{{\rule{0mm}{1mm}\\\rule{0mm}{1mm}}}

\def\emph#1{\textbf{#1}}
\def\naM{\nabla}

\newcommand{\Om}{\Omega}
\newcommand{\nuob}{\nu_\Ob}
\newcommand{\nuOb}{\nu_\Ob}
\newcommand{\de}{\delta}

  \newcommand{\nabar}{\bar\nabla}
\newcommand{\la}{\langle}
\newcommand{\ra}{\rangle}

\newcommand{\pxi}{{\nabla_i}}
\newcommand{\pxj}{{\nabla_j}}

\newcommand{\aleps}{{\alpha_\eps}}
\newcommand{\beps}{{\beta_\eps}}
\newcommand{\bepsp}{{\beta_\eps'}}
\newcommand{\bepspp}{{\beta_\eps''}}

\newcommand{\abs}[1]{\left\vert#1\right\vert} 
 
\newcommand{\eps}{\varepsilon}
\newcommand{\na}{\nabla}

\newcommand{\Ob}{\mathcal{O}}
\newcommand{\Pro}{\mathcal{P}}
\newcommand{\ob}{\mathcal{O}}

\newcommand{\beq}{\begin{equation}}
\newcommand{\eeq}{\end{equation}}
\newcommand{\beqs}{\begin{equation}}
\newcommand{\eeqs}{\end{equation}}
\newcommand{\pO}{{\partial \O}}
\newcommand{\beqa}{\begin{equation}\begin{aligned}}
\newcommand{\eeqa}{\end{aligned}\end{equation}}
\newcommand{\loc}{\text{loc}}
\newcommand{\ben}{\begin{enumerate}}
\newcommand{\een}{\end{enumerate}}
\newcommand{\Mth}{M_t\cap\left\{x^{n+2}\ge \ell\right\}}
\newcommand{\Mnullth}{M_0\cap\left\{x^{n+2}\ge \ell\right\}}

\newcommand{\beqas}{\begin{equation*}\begin{aligned}}
\newcommand{\eeqas}{\end{aligned}\end{equation*}}

\mathchardef\ordinarycolon\mathcode`\:
  \mathcode`\:=\string"8000
  \begingroup \catcode`\:=\active
    \gdef:{\mathrel{\mathop\ordinarycolon}}
  \endgroup

\newtheorem{theorem}{Theorem}[section]
\newtheorem{lemma}[theorem]{Lemma}
\newtheorem{proposition}[theorem]{Proposition}
\newtheorem{corollary}[theorem]{Corollary}

\newtheorem{definition}[theorem]{Definition}

\newtheorem{remark}[theorem]{Remark}
\newtheorem{assumption}[theorem]{Assumption}

\numberwithin{equation}{section}
\begin{document}

\title[Mean curvature flow with obstacles]{Weak solutions to Mean
  curvature flow respecting obstacles I: the graphical case}

\author{Melanie Rupflin}
\thanks{}
\address{Melanie Rupflin, Mathematisches Institut, Universit\"at Leipzig,
  Augustusplatz 10, 04109 Leipzig, Germany}
\curraddr{}
\def\LeipzigHome{@math.uni-leipzig.de}
\email{melanie.rupflin\LeipzigHome}

\author{Oliver C. Schn\"urer}
\address{Oliver C. Schn\"urer, Fachbereich Mathematik und Statistik,
  Universit\"at Konstanz, 78457 Konstanz, Germany}
\curraddr{}
\def\AmSeeHome{@uni-konstanz.de}
\email{Oliver.Schnuerer\AmSeeHome}
\thanks{}

\subjclass[2000]{53C44}

\date{\today.}

\dedicatory{}

\keywords{}

\begin{abstract}
  We consider the problem of evolving hypersurfaces by mean curvature
  flow in the presence of obstacles, that is domains which the flow is
  not allowed to enter. In this paper, we treat the case of complete
  graphs and explain how the approach of M. S\'aez and the second
  author \cite{OSMarielMCFwithoutSing} yields a global weak solution
  to the original problem for general initial data and onesided
  obstacles.
\end{abstract}

\maketitle

\tableofcontents

\section{Introduction}

Given a hypersurface in Euclidean space we investigate how one can
evolve this hypersurface by mean curvature flow if there are parts of
space, so called obstacles, that the hypersurface is forbidden from
entering.

To be more precise, let $\mathcal P$ be an open non-empty set in
Euclidean space, not necessarily connected, nor bounded or regular and
let $N_0$ be an initial hypersurface which is disjoint from $\mathcal
P$. We then would like to evolve $N_0$ by a family of hypersurfaces
$(N_t)_t$, locally described by parametrisations $F_t$, moving in
normal direction, in such a way that

\begin{enumerate}
\item $N_t$ satisfies (a weak form of) mean curvature flow 
  $$\frac{d}{dt}F=-H\nu$$ on the complement of the obstacle
  $\ol{\mathcal P}$.  
\item $N_t$ remains disjoint from the obstacle, $N_t\cap \mathcal
  P=\emptyset$.
\item In points where the hypersurface touches the (closure of the)
  obstacle, the hypersurface evolves by mean curvature flow if this
  makes the hypersurfaces lift off the obstacle, but remains
  stationary otherwise, i.\,e.{} for $p\in\bar{\mathcal P}\cap N_t$
  we would like to ask that
  $$\frac{d}{dt}F=\langle -H\nu,\nu_{\mathcal P}\rangle_+\cdot
  \nu_{\mathcal P} =(-H)_+\nu,$$ 
  where $\nu_ {\mathcal P}$ denotes the outwards pointing unit normal
  to $\partial {\mathcal P}$ (where defined) and $\langle
  a,b\rangle_+=\max(\langle a,b\rangle,0)$.
\end{enumerate}

A first approach to mean curvature flow with obstacles was carried out
by L. Almeida, A. Chambolle, and M. Novaga \cite{Almeida} who
constructed solutions based on a time-discretisation scheme for the
corresponding partial differential inequality and obtained in
particular short-time existence of $C^{1,1}$-solutions in certain
settings.  Furthermore, E. Spadaro \cite{SpadaroMeanConvex} considered
mean curvature flow with obstacles in order to investigate properties
of mean convex sets.  He used a time-discretisation to obtain a weak
mean curvature flow of Caccioppoli sets and the focus of his work is
on the properties of the limits as $t\to \infty$ of such weak
solutions.

In the present paper we show that the ideas of M. S\'aez and the
second author \cite{OSMarielMCFwithoutSing} introduced for the study
of standard mean curvature flow can be used to obtain a new approach
for mean curvature flow with obstacles that avoids the study of
singularities completely but allows us to show global existence of
weak solutions for essentially all (reasonable) initial data and
\textit{onesided} obstacles.

The basic idea of the construction is the following: Given any initial
($n$-di\-men\-sio\-nal) hypersurface $N_0\subset \R^{n+1} $ and an
obstacle $\P\subset \R^{n+1}$ we lift the problem to one dimension
higher by building complete graphs over both the obstacle and the
region enclosed by the initial hypersurface $N_0$ which contains the
obstacle, see Figure \ref{obst pic}.

We then consider the new and simpler problem of flowing a graphical
surface $M_0$ in the presence of a graphical obstacle $\Ob$ for which
we prove long-time existence of a viscosity solution. This solution of
the graphical problem is obtained as a limit of flows that do not
prohibit the penetration of the obstacle but only penalise it
appropriately. A key part of the analysis of these approximate
solutions carried out later on is to prove that they satisfy locally
uniform spatial $C^2$-estimates. This implies in particular that the
viscosity solution that we obtain is of class $C^{1,1}$ which, in view
of the analysis of the corresponding stationary problem of C. Gerhardt
\cite{CGGlobalC11}, is optimal.

Similarly to \cite{OSMarielMCFwithoutSing}, one can interpret the
projection of this graphical flow $(M_t)_t$ in $\R^{n+2}$
to $\R^{n+1}$ as a weak solution $(N_t)_t$ for the original problem of
evolving by mean curvature flow in $\R^{n+1}$ respecting the obstacle
$\P$.

After completion of our manuscript, we found out that a related
problem has been considered independently by G. Mercier and M. Novaga
\cite{MercierNovagaObstacleMCF}. While our focus is on the evolution
of complete graphs over time-dependent domains, their focus is on the
study of entire graphs that G. Mercier subsequently uses to construct
level set solutions to mean curvature flow with obstacles in 
\cite{MercierObstacleViscosity}.

In subsequent work we will relate our notion of a weak solution to
level set solutions of mean curvature flow respecting obstacles.

\section{Definition of a solution}

\begin{definition}[Initial data]
  \label{def: ini data}
  Given an open, possibly disconnected set $\Pro\subset\R^{n+1}$, we
  consider an initial hypersurface $N_0\subset\R^{n+1}$ which is
  disjoint from $\Pro\subset\R^{n+1}$ and an open, possibly unbounded
  and disconnected, set $\Omega_0\subset\R^{n+1}$, such that
  \[\partial\Omega_0=N_0\quad\text{and}\quad\Pro\subset\Omega_0.\]

  For the lifted problem in $\R^{n+2}$ we then consider initial data
  consisting of an obstacle $\Ob$ and an initial hypersurface $M_0$
  with the following properties.
  \begin{enumerate}[(i)]
  \item The obstacle $\Ob\subset\R^{n+2}$ is given as
    \[\Ob=\left\{\left(\hat x,x^{n+2}\right)\in\R^{n+2}\colon
      x^{n+2}<\psi(\hat x)\right\}\] for a function $\psi\in
    C^{1,1}_{\loc}(\P)$ which is proper and bounded above.\par In
    particular, $\psi(\hat x)\to-\infty$ for $\hat x\to\partial\Pro$
    or $|\hat x|\to\infty$.
  \item[(ii)] The initial hypersurface $M_0\subset\R^{n+2}$ is given
    as \[M_0=\graph u_0\] for a locally Lipschitz function
    $u_0\colon\Omega_0\to\R$ which is proper, bounded above and
    fulfils
    \[u_0\ge\psi\quad\text{in}\quad \Pro\subset\Omega_0.\]
  \end{enumerate}
\end{definition}

\begin{figure}[htb] 
 \includegraphics[height=5cm]{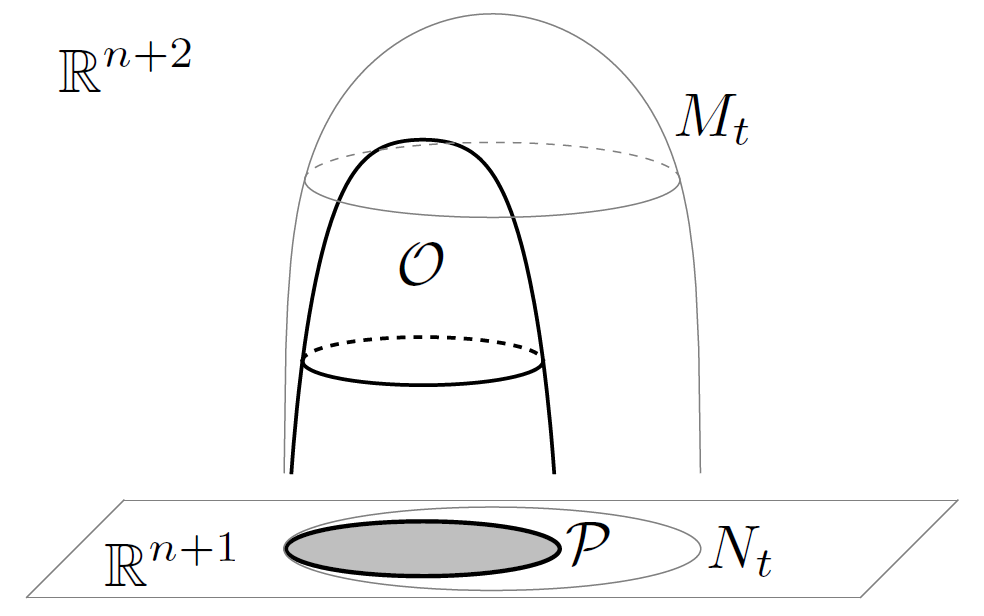}
 \caption{Graphical initial surface $M_0$ and obstacle $\partial\O$ in
   $\R^{n+2}$ associated with the original data $N_0$ and $\P$.} 
 \label{obst pic} 
\end{figure} 

We remark that there is no need to impose any regularity assumptions
on either $\Pro$ or $\partial \Omega_0$ in order to obtain such lifted
initial data $\Ob$ and $M_0$. Furthermore, $\O$ can and will be chosen
so that $\partial\O$ has uniformly bounded second fundamental form if
$\partial\P$ has uniformly bounded second fundamental form and a
tubular neighbourhood with thickness uniformly bounded below. An
analogous statement holds for $M_0$ and $N_0=\partial\Omega_0$.

We adapt the definition of a solution to graphical mean curvature flow
from \cite{OSMarielMCFwithoutSing} to the situation with obstacles. We
follow the convention that the obstacle lies below the solution, see
e.\,g.{} \cite{CGGlobalC11}, and therefore have to reflect the setting
in \cite{OSMarielMCFwithoutSing}. In particular the
evolving hypersurface $M_t=\graph u(\cdot,t)|_{\Omega_t}$ will be
represented by a pair $(\Omega,u)$, where
$\Omega\subset\R^{n+1}\times[0,\infty)$ is a subset of space-time,
$u(x,t)$ is defined for $(x,t)\in\Omega$ and $\Omega_t\subset\R^{n+1}$
is a time-slice of the space-time domain $\Om$ as defined below.  We
refer to \cite{OSMarielMCFwithoutSing} for a more in depth discussion
of the motivation behind the definition.

In the following definition we use standard notation: $H$ denotes the
mean curvature of $M_t$ and $v=\langle\nu,e_{n+2}\rangle^{-1}$. For
details we refer to Section \ref{nota sec}.
\begin{definition}[graphical mean curvature flow with obstacle]
  \label{def:gr-mcf-obst}
  \neueZeile
  \begin{enumerate}[(i)]
  \item\label{def sol domain} \textbf{Domain of definition:} Let
    $\Omega\subset\R^{n+1}\times[0,\infty)$ be a (relatively) open
    set. Set $\Omega_t:=\pi_{\R^{n+1}}\left(\Omega\cap\left(\R^{n+1}
        \times\{t\}\right)\right)$, where
    $\pi_{\R^{n+1}}\colon\R^{n+2}\to\R^{n+1}$ is the orthogonal
    projection to the first $n+1$ components. We require that
    $\mathcal P\subset\Omega_t$ for every $t\in[0,\infty)$.
  \item \textbf{The solution:} A function $u\colon\Omega\to\R$ is
    called a solution to graphical mean curvature flow in $\Omega$
    with initial value $u_0\colon\Omega_0\to\R$ and obstacle
    $(\mathcal P,\psi)$ or $\mathcal O$, if $u\in C^0_\loc(\Omega)$
    satisfies
    \begin{equation}
      \label{MCF}
      \begin{cases}
        \min\left\{\dot u-\sqrt{1+|Du|^2}\cdot
          \divergenz\left(\frac{Du}{\sqrt{1+|Du|^2}}\right),
          u-\psi\right\}=0
        &\text{in }\Omega,\\
        u(\cdot,0)=u_0&\text{in }\Omega_0,
      \end{cases}
    \end{equation}
   in the viscosity sense. 
  \item\label{def sol max} \textbf{Maximality condition:} A function
    $u\colon\Omega\to\R$ fulfils the maximality condition if $u\le c$
    for some $c\in\R$ and if
    $u|_{\Omega\cap\left(\R^{n+1}\times[0,T]\right)}$ is proper for
    every $T>0$. An initial value $u_0\colon\Omega_0\to\R$,
    $\Omega_0\subset\R^{n+1}$, is said to fulfil the maximality
    condition if $w\colon\Omega_0\times[0,\infty)\to\R$ defined by
    $w(x,t):=u_0(x)$ fulfils the maximality condition.
  \item \textbf{Singularity resolving solution:} $(\Omega,u)$, 
  or equivalently $(M_t)_{t\ge0}$ given by $M_t=\graph
    u(\cdot,t)|_{\Omega_t}\subset\R^{n+2}$, 
    is called a singularity resolving solution to mean curvature flow
    respecting the obstacle $\O$ if the conditions
    \eqref{def sol domain}-\eqref{def sol max} are fulfilled.
  \end{enumerate}
\end{definition}
 
The formulation involving the minimum in \eqref{MCF} is a standard
description for viscosity solutions to obstacle problems
cf. \cite[Example 1.7]{UsersGuideViscosity}.  We remark that the above
definition immediately implies that $u\ge\psi$ and that $\dot
u+\sqrt{1+|Du|^2}\cdot H=0$ in the viscosity sense wherever
$u>\psi$. Furthermore

\vspace{0.2cm}
\begin{remark}
  For a $C^{2;1}$-function $u$, the equation \eqref{MCF} is fulfilled
  if and only if $u$ is a solution to
  \[
  \begin{cases}
    \dot u=\sqrt{1+|Du|^2}\cdot\divergenz
    \left(\fracd{Du}{\sqrt{1+|Du|^2}}\right) \equiv -v\cdot H
    &\text{in
    }\Omega\setminus((\Omega_0\times\{0\})\cup\Gamma),\\
    \dot u=v\cdot(-H)_+ &\text{in
    }(\Omega\setminus(\Omega_0\times\{0\}))\cap\Gamma,\\
    u(\cdot,0)=u_0\geq \psi&\text{in }\Omega_0,
  \end{cases}
  \]
  where \[\Gamma:=\{(x,t)\in\Omega\colon u(x,t)=\psi(x)\}\] is the
  contact set between the evolving hypersurface and the obstacle.
\end{remark}

In $C^{2;1}$ and more generally for parabolic H\"older spaces, the
first exponent refers to regularity in spatial and the second in time
directions. 

\section{Main results and overview of the proof}

We prove
\begin{theorem}\label{thm1}
  Let $\Ob$, $\Om_0$ and $u_0$ be an obstacle and an initial datum as 
  in Definition \ref{def: ini data}.  Then there exists a singularity
  resolving solution $(\Om, u)$ with
  \[u\in C^{1,1;0,1}_{\loc}(\Omega\setminus(\Omega_0\times\{0\})) \cap
  C^0_{\loc}(\Omega)\] of mean curvature flow respecting the obstacle
  $\Ob$ for all times.
 
  Furthermore, the evolving surface $M_t:=\graph u(\cdot,t)$ is
  controlled in halfspaces of the form $\left\{x^{n+2}>\ell \right\}$ for
  arbitrary $\ell\in\R$ in the sense that
  $v=\langle\nu,e_{n+2}\rangle^{-1}$ and the second fundamental form
  $A$ of $M^\ell_t :=M_t\cap \left\{x^{n+2}>\ell\right\}$ satisfy
  \begin{equation}
    \label{est:thm-v-A-bound-lip-ini-data}
    \Vert v\Vert_{L^\infty\left(M_t^\ell\right)}+\sqrt t\cdot\Vert
    A\Vert_{L^\infty\left(M_t^\ell\right)}\le C(u_0,\O,\ell). 
  \end{equation}
  Furthermore, if the initial surface $M_0$ is $C^{1,1}_{\loc}$, then
  $M^\ell_t$ has uniformly controlled second fundamental form $\Vert
  A\Vert_{L^\infty\left(M^\ell_t\right)}\le C(u_0,\O,\ell)$ up to time
  $t=0$.\par In addition, for positive times, $u$ is smooth away from
  the contact set. 
\end{theorem}

\begin{remark}
  \neueZeile
  \begin{enumerate}[(i)]
  \item The regularity statement of Theorem \ref{thm1} can be seen as
    the analogue of C. Gerhardt's $C^{1,1}$-regularity result
    \cite{CGGlobalC11} for solutions of the stationary obstacle
    problem. The simple example of a rope spanned over a circle
    illustrates in both cases that the spatial $C^{1,1}$-regularity is
    optimal.
  \item As $C^{1,1}$-functions are twice differentiable almost
    everywhere, the second fundamental form is defined almost
    everywhere and the above $L^\infty$-bounds on the second
    fundamental form and the gradient are equivalent to local
    $C^{1,1}$-bounds. 
  \end{enumerate}
\end{remark}

As it is of interest to consider not only complete but also entire
graphs, we prove additionally
\begin{theorem}
  \label{thm2}
  Let $u_0\colon\R^{n+1}\to\R$ be bounded and Lipschitz
  continuous. Assume that $u_0$ is constant outside a compact subset
  of $\R^{n+1}$.  Let $\psi\colon\R^{n+1}\to\R$ be a function
  describing an obstacle as in Definition \ref{def: ini data}. Assume
  furthermore that $u_0\ge\psi$.  Then there exists a uniformly
  continuous viscosity solution $u\colon\R^{n+1}\times[0,\infty)\to\R$
  of mean curvature flow with obstacle
  \[\min\left\{\dot u-\sqrt{1+|Du|^2}\cdot
    \divergenz\left(\frac{Du}{\sqrt{1+|Du|^2}}\right),
    u-\psi\right\}=0\] with $u(\cdot,0)=u_0$.  Furthermore
  \[\Vert u\Vert_{L^\infty(M_t)} +\Vert v\Vert_{L^\infty(M_t)} +\sqrt
  t\cdot\Vert A\Vert_{L^\infty(M_t)}\le C(u_0,\psi).\] 
\end{theorem}

Theorem \ref{thm2} could be used to construct viscosity solutions for
mean curvature flow with obstacles based on the level set
approach. Such solutions were recently constructed in
\cite{MercierObstacleViscosity}.

Of course, in the absence of an obstacle, this result is a
special case of \cite{EckerHuiskenInvent}.
\medskip

The approach we use to construct a solution of mean curvature flow
with obstacles in the graphical setting is by penalisation. We obtain
the desired viscosity solution as a limit of solutions to problems
which allow a penetration of the obstacle, but penalise it by stronger
and stronger normal vector fields trying to push the hypersurface back
out of the obstacle.

More precisely, we fix a function $\beta\in C^{\infty}(\R,
[0,\infty))$, supported in $(-\infty,0]$ with $\beta''$
non-increasing, and thus in particular satisfying $\beta''>0$ whenever
$\beta>0$, and consequently also $\beta'<0$.

We furthermore define $\dist_\pO$ to be the signed distance function
to the boundary of $\O$ chosen so that $\dist_\pO$ is negative in
$\O$. 

Given $\eps>0$ we then consider the flow 
\begin{equation}
  \label{eps fluss}
  \tdt F=-(H-\alpha_\epsilon)\cdot\nu=\Delta F+\aleps \nu,
\end{equation}
where 
$$\aleps(p):=\beps(\dist_\pO(p)), \quad
\beta_\eps(s)=\beta\left(\frac{s}\eps\right)$$ and 
where $\Delta$ is the Laplacian on the evolving submanifold so that
$-H\nu=\Delta F$.

We stress that our penalisation depends on the Euclidean distance to
$\partial\O\subset\R^{n+2}$ and not on the graphical one, i.\,e.{} not
on $u(x,t)-\psi(x)$. This feature of the con\-struc\-tion is crucial
in order to be able to deal with complete graphs over possibly bounded
domains.

While solutions to the penalised flow can sink into the obstacle, we
shall show in Section \ref{sect:penetration} that the depth of this
penetration is of order $O(\epsilon)$.  In Section \ref{section:C1} we
shall then prove that the gradient function of these approximate
solutions is bounded uniformly in time and locally in space. Similar
$C^{1,1}$-estimates will be deduced in the following Section
\ref{section:A}. We stress that these estimates are independent of the
parameter $\eps$ of the penalisation which thus immediately gives
$C^{1,1}$ regularity also for our viscosity solution of mean curvature
flow with obstacles which we obtain in the limit $\eps\searrow 0$, see
Section \ref{sect:proofs}.

While we will state and prove these results only for
smooth obstacles, all the estimates derived in Sections
\ref{section:C1} and \ref{section:A} depend only on the local
$C^2$-norm of $\psi$, so we are able to reduce the proof of Theorems
\ref{thm1} and \ref{thm2} to the case of smooth obstacles and an
approximation argument carried out later on in Section
\ref{sect:proofs}. In particular we will assume from now on that
$\psi$ is smooth unless stated otherwise.

\section{Notations and geometry of submanifolds}
\label{nota sec}
We use $F=F(x,\,t)=\left(F^\alpha\right)_{1\le\alpha\le n+2}$ to
denote the time-dependent embedding vector of a manifold $M^{n+1}$
into $\R^{n+2}$ and $\dt F=\dot F $ for its total time derivative.
We set $M_t:=F(M,\,t)\subset\R^{n+2}$ and will often identify an embedded
manifold with its image. We will assume that $F$ is smooth.  We assume
furthermore that $M^{n+1}$ is smooth and orientable.  The embedding
$F(\cdot,\,t)$ induces a metric $g=(g_{ij})_{1\le i,\,j\le n+1}$ on
$M_t$. We denote by $\na$ the Levi-Civit\'a connection on $(M_t,g(t))_t$
and the induced bundles while we write $\bar\na$ for the gradient on
the ambient space $\R^{n+2}$.

We choose $\nu=\left(\nu^\alpha\right)_{1\le\alpha\le n+2}$ to be the
upward pointing unit normal vector to $M_t$ at $x\in M_t$.

The second fundamental form $A$ is then characterized through the
Gau\ss{} equation
\begin{equation}\label{Gauss formula}
  \nabla_i\nabla_j F=-A_{ij}\nu
\end{equation}
or, equivalently, the Weingarten equation
\[\nabla_i\nu=A_{il}g^{lk}\nabla_kF=A^k_i\nabla_kF.\]
Here and in the following, we raise and lower indices using the metric
and its inverse $\big(g^{ij}\big)$ and utilize the Einstein summation
convention to sum over repeated upper and lower indices.

Throughout the paper, Latin indices range from $1$ to $n+1$ and refer
to geometric quantities on the hypersurface, while Greek indices refer
to the components in fixed Euclidean coordinates in the ambient space
$\R^{n+2}$.

We define the mean curvature $H$ by $H=g^{ij}A_{ij}$ and compute the
norm of the second fundamental form through
$|A|^2=A_{ij}g^{jk}A_{kl}g^{il}$.

Finally, given a function $f$ defined on the ambient space $\R^{n+2}$
we write $\naM f$ for the derivative of $f\vert_{M_t}$ on $M_t$ which
can equivalently be computed as the projection
$$\naM f=P_{TM}\left(\bar \na f\right)=\bar\na f-\left\la\bar\na
  f,\nu\right\ra \nu,$$ of the ambient gradient to the tangent space
of the evolving hypersurface $M_t$. Here we use in the last equality
that this orthogonal projection $P_{TM}:\R^{m+2}\to T_pM_t$, $p\in M$,
can be expressed in terms of the normal as $P_{TM}(X)=X-\la X,\nu\ra
\nu$, where $\la \cdot,\cdot \ra$ denotes the Euclidean inner product
on $\R^{n+2}$.  Furthermore we will consider the gradient 
 $\naM f(p,t)$ of functions $f$, be they defined on all of
$\R^{n+2}$ or only on $M_t$, as a
vector in either $T_pM_t$ or in $\R^{n+2}$ as convenient and without
changing the notation. Similarly, we will evaluate geometric
quantities either at $(x,t)\in M\times[0,\infty)$ or at $p=F(x,t)\in
M_t\subset\R^{n+2}$.

As the topology of our solutions may change, we only require that
solutions to \eqref{eps fluss} are parametrised over a base manifold
$M$ locally in space and time.

We shall also use that the Gau{\ss} equation allows us to express
the Riemannian curvature tensor of the surface in terms of the second
fundamental form
\[R_{ijkl}=A_{ik}A_{jl}-A_{il}A_{jk}.\]

Throughout the paper, expressions like $\nabla_i\nabla_jA_{kl}$ are to
be understood as first computing the covariant derivatives of the
tensor $A$ and then evaluating it in the indicated directions of the
standard basis vector fields.

\section{Evolution equations} 
In this section we collect the evolution equations of the various
geometric quantities such as gradient function, second fundamental
form, etc. As the corresponding formulas for mean curvature flow, and
more generally for graphical flows moving in normal direction, are
well known, see \cite{EckerBook,CGCPBook,HuiskenPolden}, we will
mainly analyse the influence of the penalisation $\aleps$.

We remark that the distance function $\dist_\pO$ as well as its level
sets are $C^2_\loc$ in a neighbourhood of $\partial \Ob$ and that
throughout this section we shall only consider points which, if they
are in $\O$, are contained in such a neighbourhood. We will later
justify this assumption as a consequence of Lemma
\ref{lemma:penetration}.

To begin with, we define the height function of the evolving
hypersurface by \[U:=\langle F,e_{n+2}\rangle.\] For graphical
hypersurfaces, the penalised flow \eqref{eps fluss} can be rewritten
in terms of $U$ as \[\tdt U-\Delta
U=\alpha_\epsilon\langle\nu,{e_{n+2}}\rangle=\frac{\aleps}{v},\] $v$
the gradient function introduced above.

For a family of hypersurfaces moving with normal velocity $f$,
\beq \label{eq:general-flow} \frac{d}{dt}F=-f\cdot \nu,\eeq $f$ any
function defined on the evolving hypersurfaces, it is well known that
the metric evolves by $\dt g_{ij}=-2fA_{ij}$ which becomes
\begin{equation}
  \label{eq:evol-g}
  \tdt g_{ij}=-2(H-\alpha_\epsilon)A_{ij}  
\end{equation}
in our case.  The normal evolves by
$$\tdt \nu=\na f,$$
so using the identity
$$\Delta \nu=-\abs{A}^2\nu+\na H,$$
valid for arbitrary hypersurfaces in Euclidean space, we obtain in
this more general context of \eqref{eq:general-flow} that
\beq \label{eq:ddtnu-general}
\left(\tdt-\Delta\right)\nu=\abs{A}^2\nu+\na(f-H), \eeq which for our
flow translates to

\begin{lemma}
  For hypersurfaces evolving according to \eqref{eps fluss}, 
  $\nu$ fulfills
  \begin{align*} \tdt\nu-\Delta\nu=&\,|A|^2\nu-\naM\alpha_\epsilon,
    \intertext{or, equivalently, written out in local coordinates }
    \tdt\nu^\beta-\Delta \nu^\beta=&\,|A|^2\nu^\beta
    -\ol\nabla_\gamma\alpha_{\epsilon}\nabla_iF^\gamma
    g^{ij}\nabla_jF^\beta.
  \end{align*}
\end{lemma}

With $\aleps$ given by $\aleps=\beps\circ\dist_\pO$, its derivative in
a point $p\in M_t\cap \O$ is determined in terms of $\nu_\Ob=\bar \na
\dist_\Ob$ (where defined) which describes the outwards unit normal to
the level set
$$\pO_\delta:=\{y\in\R^{n+2 }: \dist_\pO(y)=-\de\}$$
which contains $p$. Namely, 
\beq 
\label{eq:Daleps}
\naM\aleps=\bepsp\cdot \naM\dist_\pO =\bepsp \cdot P_{TM}(\nu_\Ob)
=\bepsp\cdot(\nu_\Ob-\langle\nu_\Ob,\nu\rangle\nu),\eeq or
equivalently, working in local coordinates, $\nabla_j\alpha_\epsilon
=\beta_\epsilon'\cdot\langle\nu_\O,\nabla_jF\rangle$.
  
Outside of $\O$, the derivative
of $\alpha_\epsilon$ vanishes.

For graphical solutions of \eqref{eps fluss}, or more generally of
\eqref{eq:general-flow}, we then consider the 'gradient function'
$v$ defined by $v=\langle\nu,{e_{n+2}}\rangle^{-1}$ which, by
\eqref{eq:ddtnu-general}, satisfies 
\begin{align*}
  \left(\tdt-\Delta\right)v=&\,-\langle\nu,{e_{n+2}}\rangle^{-2}\cdot
  \left\la\left(\tdt-\Delta\right)\nu,e_{n+2}\right\ra
  -2\langle\nu,{e_{n+2}}\rangle^{-3}
  \abs{\na \langle\nu,{e_{n+2}}\rangle}^2\\
  =&\, -v^2\cdot \left[\langle\nu,{e_{n+2}}\rangle
    \abs{A}^2+\left\la\na(f-H),e_{n+2}\right\ra\right]
  -2v^{3}\abs{\na \left(v^{-1}\right)}^2\\
  =&\, -\abs{A}^2\cdot v-2\frac{\abs{\na
      v}^2}{v}-\la\na(f-H),e_{n+2}\ra\cdot v^2.
\end{align*}
We shall later use that we can express $\na v$ in terms of the second
fundamental form as \beq \label{eq:deriv-v}
\na_iv=-v^2\la\na_i\nu,e_{n+2}\ra=-v^2A_i^k\la\na_kF,e_{n+2}\ra
=-v^2A_i^k\na_k U.  \eeq but for now only need the conclusion that
\begin{lemma}\label{v evol}
  For graphical hypersurfaces evolving according to \eqref{eps fluss},
  the gradient function $v=\langle\nu,{e_{n+2}}\rangle^{-1}$ fulfills
  \begin{equation}
    \label{eq v evol}
    \tdt v-\Delta v=-|A|^2v -\tfrac 2v|\nabla v|^2 +v^2\left\langle\naM
    \alpha_\epsilon,{e_{n+2}}\right\rangle.
  \end{equation}
\end{lemma}

Compared with standard mean curvature flow we thus obtain an
additional term that contains a derivative of the penalty function and
which may thus become arbitrarily large in the limit $\eps\searrow0$. 

However, as illustrated in Figure \ref{gradient pic}, in a point where
the evolving surface is `steeper' than the obstacle, the penalisation
helps to reduce $v$, because $\alpha_\epsilon$ grows with increasing
(negative) distance to $\partial\O$.

\begin{figure}[htb] 
 \includegraphics[height=4cm]{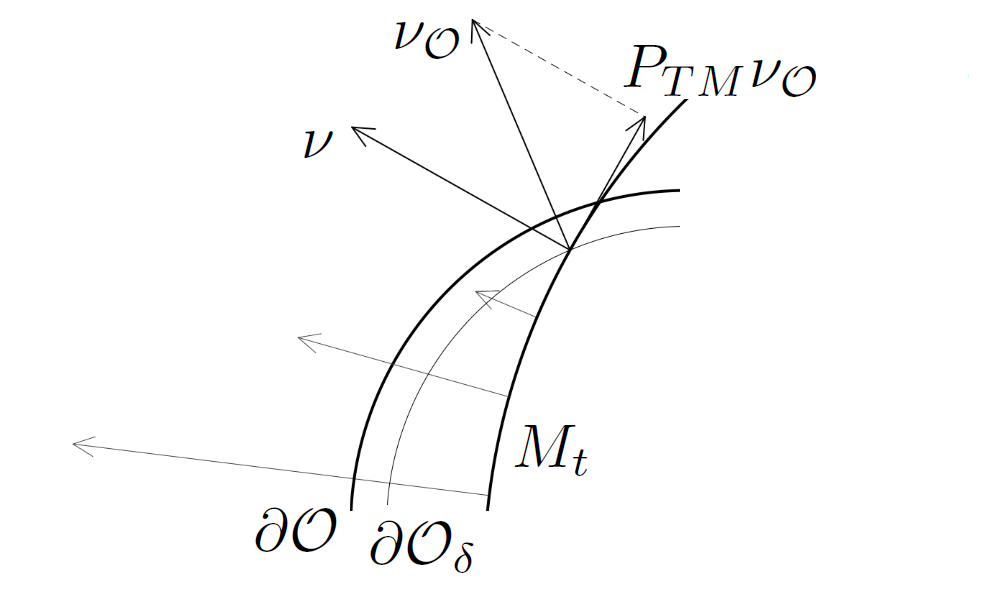}
 \caption{Penalising vectorfield and normals in a point where $v\geq
   v_\nu$.}
 \label{gradient pic} 
\end{figure} 

More precisely, we obtain
\begin{remark}\label{rem:sign nabla alpha eta}
  Given a point $p\in \partial\Ob_\delta$ in a neighbourhood of which
  $\partial \Ob_\de$ is a $C^1$ graph we let
  $v_\Ob:=\langle\nu_{\O},{e_{n+2}}\rangle^{-1}$ be the gradient
  function (of the level sets) of the obstacle.  Then at each point
  $p\in\Ob\cap M_t$ where
  $$v\ge v_\Ob,$$
  we have
  $$\left\langle\naM\alpha_\epsilon,{e_{n+2}}\right\rangle\leq 0.$$
\end{remark}

\begin{proof}
  Since both the evolving hypersurface and the level sets of the
  obstacle are graphical and thus $v,v_\Ob$ are well defined and
  positive we can use \eqref{eq:Daleps} to compute
  \[ \left\langle\naM\alpha_\epsilon,{e_{n+2}}\right\rangle=\bepsp
  \left(\la \nu_\Ob,e_{n+2}\ra -\langle\nu_\Ob,\nu\rangle\cdot \la
    \nu,e_{n+2}\ra\right)=\bepsp\cdot
  \left(\frac{1}{v_\Ob}-\frac{\langle\nu_\Ob,\nu\rangle}{v}\right)\]
  which gives the claim as $\bepsp\leq 0$.
\end{proof}

We finally turn to the evolution equation satisfied by the norm of the
second fundamental form.

It is well known that $\abs{A}^2$ evolves along a normal flow
\eqref{eq:general-flow} according to
$$\tdt\abs{A}^2=2fA^k_iA^i_jA^j_k+2A^{ij}\na_i\na_j f$$ 
as well as that
$$\Delta\abs{A}^2 =2A^{ij}\na_i\na_j H+2\abs{\na
  A}^2+2HA^k_iA^i_jA^j_k-2\abs{A}^4.$$
This implies the general formula
$$\left(\tdt-\Delta\right)\abs{A}^2
=-2A^{ij}\na_i\na_j(H-f) -2(H-f)A^k_iA^i_jA^j_k
+2\abs{A}^4-2\abs{\na A}^2$$ which in our case becomes

\begin{lemma}
  For hypersurfaces evolving by the penalised flow \eqref{eps fluss},
  the norm of the second fundamental form fulfils
  \beqa \label{eq:evolA-1}
  \left(\tdt-\Delta\right)\abs{A}^2&=-2\abs{\na A}^2+2\abs{A}^4
  -2\aleps A^k_iA^i_jA^j_k -2\na_i\na_j \aleps A^{ij}.
  \eeqa
\end{lemma}

The last term in this equation, given as the covariant derivative of
the vector field $\naM\aleps \in \Gamma(TM)$, needs to be analysed
carefully as it contains a second order derivative of the penalty
function. As such it can be of order $\eps^{-2}$ at points in the
obstacle which might be reached by the evolving hypersurface, compare
also Section \ref{sect:penetration}.

The second covariant derivative of the penalisation function
$\alpha_\epsilon$ is given by
\beqa \label{eq:D2aleps} \nabla_i\nabla_j
  \aleps&=\na_{i}(\nabar_{\pxj F}\aleps)
=\na_{i}\big((\bepsp\circ\dist_\pO)\cdot\langle \nu_\Ob,\pxj
F\rangle\big)\\
&=\bepspp\cdot\langle\nu_\Ob,\pxi F\rangle \cdot \langle \nu_\Ob,\pxj
F\rangle +\bepsp\cdot\left\langle \bar\na_{\pxi F}\nu_\Ob,\pxj
  F\right\rangle +\bepsp \langle\nu_\Ob,\pxi\pxj F\rangle.  \eeqa

The last term in this formula is
given by
$$\bepsp \langle\nu_\Ob,\pxi\pxj F\rangle=- \bepsp A_{ij}\la
\nuob,\nu\ra.$$ 

For a better understanding of the penultimate term in
\eqref{eq:D2aleps}, we choose an orthonormal basis $(e_a)$ of the
tangent space to the level set $\pO_\delta $ which contains our point
$p$ and write
$$\pxi F=\langle\pxi F,e_a\rangle\delta^{ab}e_b+\langle\pxi
F,\nu_\Ob\rangle \cdot \nu_\Ob.$$ 
In the resulting formula 
$$\nabar_{\pxi F}\nu_\Ob=\langle\pxi F,e_a\rangle \delta^{ab}\cdot
\nabar_{e_b}\nu_\Ob +\langle \pxi F,\nu_\Ob\rangle
\nabar_{\nu_\Ob}\nu_\ob,$$ 
the first term contains
$$\nabar_{e_b}\nu_\Ob=\left\la \nabar_{e_b}\nu_\Ob,e_c\right\ra
\de^{cd} e_d=A_{bc}^\Ob \de^{cd}e_d,$$ 
the (locally) bounded second fundamental form of the obstacle (or
rather its level set $\pO_\de$), while the second term can be seen to
vanish identically; indeed since $\abs{\nabar\dist_\pO}^2\equiv 1$ we
obtain for every $\gamma=1,\ldots,n+2$
\begin{align*}
  \left(\nabar_{\nu_\Ob}\nu_{\Ob}\right)^\gamma =&\,\nu_\Ob^\eta
  \tfracp{}{y^\eta} \nu_\Ob^\gamma =\sum_{\eta=1}^{n+2}
  \big(\tfracp{}{y^\eta}\tfracp{}{y^\gamma}\dist_\pO\big)
  \tfracp{}{y^\eta}\dist_\pO =\tfrac12
  \tfracp{}{y^\gamma}\left\lvert\nabar \dist_\pO\right\rvert^2=0.
\end{align*}
Thus we can express the coefficient in the penultimate term in
\eqref{eq:D2aleps} 
\begin{equation}
\label{def:AO-tilde}\left\la \nabar_{\pxi F} \nuOb,\pxj F\right\ra
=A_{bc}^\Ob 
\de^{cd}\de^{ab} \la e_a,\pxi F\ra\cdot \la e_d,\pxj F\ra =:\tilde
A^\Ob_{ij},
\end{equation}
$i,j \in\{1,..,n+1\}$,
in terms of a tensor $\tilde A^\Ob$ which is controlled by  $A^\Ob$.

All in all, the derivative of the penalisation is thus given by
\beqa \label{eq:D2aleps2}
\na_i \na_j \aleps
&=\bepspp\la\nu_\Ob,\pxi F\ra\cdot\la\nu_\Ob,\pxj F\ra +\bepsp\tilde
A_{ij}^\Ob-\bepsp A_{ij}\la \nu_\Ob,\nu\ra 
\eeqa
which, once inserted into \eqref{eq:evolA-1}, results in 
\begin{lemma}
  For hypersurfaces evolving by the penalised flow \eqref{eps fluss}, 
  we have 
  \begin{align}
    \left(\tdt-\Delta\right)\abs{A}^2=&\,-2\abs{\na
      A}^2+2\abs{A}^4-2\aleps
    A^k_iA^i_jA^j_k\nonumber\\
    \label{eq:evolA2}&\,-2\bepspp\la \nu_\Ob,\pxi F\ra \cdot \la
    \nu_\Ob,\pxj F\ra A^{ij} -2\bepsp\tilde A_{ij}^\Ob A^{ij} +2\bepsp
    \la\nu,\nu_\Ob\ra \abs{A}^2,
  \end{align}
where $\tilde A_{ij}^\Ob$ is given by
  \eqref{def:AO-tilde}.
\end{lemma}

Contrary to the evolution equation for the gradient function, we
cannot expect the additional terms to have a sign, so deriving suitable
a priori bounds on the second fundamental form will be one of the main
tasks in the analysis of the penalised flow \eqref{eps fluss}. As we
shall see, we can deal with this problem by considering a modified
second fundamental form quantity which depends also on the penalty
function itself.

For this we shall in particular need the evolution equation of the
penalty function itself which is given by
\begin{lemma}
  For hypersurfaces evolving by the penalised flow \eqref{eps fluss},
  we have \beq
\label{eq:evol-aleps}
\left(\tdt-\Delta_M\right)\aleps=\bepsp\aleps\la
\nu_\Ob,\nu\ra-\bepspp\abs{P_{TM}\nu_\Ob}^2-\bepsp\tilde A_{ij}^\Ob
g^{ij}.\eeq 
\end{lemma}
Observe that the second term of this evolution equation gives a strong
negative contribution (scaling as $\eps^{-2}$) in points of the
obstacle where the evolving surface is not tangential to the level
sets of the obstacle.

\begin{proof}
  The formulas for the derivatives of the penalty function, see
  \eqref{eq:D2aleps2} and the formula following \eqref{eq:Daleps},
  immediately imply that 
  \begin{align*} (\tdt-\Delta_M)\aleps=&\,\bepsp \left\la \nu_\Ob,
      \left(\tdt-\Delta_M\right)F\right\rangle -\bepspp\cdot \la
    \nu_\Ob,\pxi F\ra\cdot \la\nu_\Ob,\pxj F\ra
    g^{ij}
    -\bepsp \cdot \tilde A_{ij}^\Ob g^{ij}\\
    =&\, \bepsp\aleps\la
    \nu_\Ob,\nu\ra-\bepspp\abs{P_{TM}\nu_\Ob}^2-\bepsp\tilde
    A_{ij}^\Ob g^{ij}
  \end{align*}
  as claimed.
\end{proof}

\section{Estimates on the depth of penetration
}\label{sect:penetration}
We shall later obtain the desired viscosity solution as limit of
solutions to Dirichlet problems for \eqref{eps fluss} to be solved on
larger and larger balls $B_R(0)$ where we will truncate the initial
map $u_0$ at levels $L\ll0$.  In this context we shall always assume
that $R$ is sufficiently large so that $\psi< L$ outside $B_R(0)$.

We prove the following bound for the amount that the evolving
hypersurface can sink into the obstacle.
\begin{lemma}\label{lemma:penetration}
  For any height $\ell\in\R$, there exists a number
  $C_0(\ell)\in(1,\infty)$ with the following property:

  For any $L\in(-\infty,\ell)$ and $R>0$ as above, there exists
  $\eps_0(L)>0$, such that for $0<\epsilon\le\epsilon_0(L)$ any
  hypersurface $M_t=\graph(u^L_{\epsilon,R}(\cdot, t))$ evolving
  according to
  \begin{equation}
  \label{eq:eps L graph}
  \begin{cases}
    \dot u =\sqrt{1+|Du|^2}\cdot\left(\divergenz
      \left(\fracd{Du}{\sqrt{1+|Du|^2}}\right) +\alpha_\epsilon\right)
    &\text{in }B_R(0)\times[0,T),\\ u=L&\text{on }\partial
    B_R(0)\times[0,T),\\ u(\cdot,0)=u_0\ge\max\{\psi,L\}&\text{in
    }B_R(0),
  \end{cases}
  \end{equation}
  satisfies \beq\label{est:pen-depth} \dist_{\partial\O}(p)\ge
  -C_0(\ell)\cdot \epsilon \eeq in any point $p\in
  M_t\cap\left\{x^{n+2}\geq \ell\right\}$ and for all times
  $t\in[0,T)$.
\end{lemma}

We stress that the level $L$ at which we truncate the hypersurface
only determines the range of admissible parameters $\eps$, but that the
bounds on the depth of penetration on $\left\{x^{n+2}>\ell\right\}$
are \textit{independent} of $L$. To achieve this, we shall compare the
evolving hypersurface with deformed level sets to $\partial\O$ of the
following type.

\begin{lemma}
  \label{lemma:comp-surf-zwei-ff}
  Given any function $f_0\in C_{\loc}^2(\R,\R_+)$ 
  and any number $\eps>0$, we let 
    $$\mathcal S_\epsilon:=\left\{p\in \R^{n+2}\colon
      \dist_{\pO}(p)=-\epsilon f_0(p)\right\}.$$
    
    Then for any $L>-\infty$ and $\delta>0$,
  there exists a number $\eps_1=\eps_1(L,f_0,\delta)>0$ such
  that for any $|\eps|<\eps_1$ the hypersurfaces
  $$S_\eps\cap \{x^{n+2}>L\}$$
 are of class $C^2$ with second fundamental form bounded by 
  $$\abs{A^{\mathcal S_\epsilon}}(p)\leq  (1+\delta)\cdot 
  \big| A^\Ob\big|(p) +\delta$$ for any $p\in\{x^{n+2}>
  L\}\cap\mathcal S_\epsilon$, where $ A^\O$ denotes the
  second fundamental form of the level set of $\dist_{\partial\O}$ that
  contains $p$.
  
  In particular, there is a number $\eps_2>0$ depending only on 
  $L$, the function $f_0$ and 
  on $\sup\limits_{\partial \Ob\cap\{x^{n+2}>L-1\}}\abs{A^\Ob}$ so that 
  $$\abs{A^{\mathcal S_\epsilon}}(p)\leq  2\cdot \Big(\sup_{B_1(p)\cap
    \pO} \abs{A^\Ob}+1\Big)$$ for $p\in \mathcal{S}_\eps\cap
  \{x^{n+2}>L\}$ and $\abs{\eps}<\eps_2$.
\end{lemma}
\begin{proof}
  We first recall that given any function $w\in C^2(\R^{n+2})$ and a
  point $p_0\in\R^{n+2}$ such that $Dw(p_0)\neq0$ one can compute the
  second fundamental form of the (locally $C^2$-) hypersurface
  \[\left\{p\in\R^{n+2}\colon w(p)=w(p_0)\right\}\]
  by 
  \[\tilde A^w(p_0)
  =\pm\frac{D^2w(p_0)}{|Dw|(p_0)}.\]
  
  In our case $\mathcal S_\epsilon=\left\{p\in\R^{n+2}\colon
    w(p)=\epsilon\right\}$ is such a level set for $w:=\frac {\rho}
  {f_0}$, where we write for short $\rho\equiv \dist_\pO$.
    
  Observe that the second term on the right-hand side in
  $$\bar \na w=\,\frac1{f_0}\bar \na \rho-\frac {\rho}{f_0^2}\na
  f_0=\frac 1{f_0}\nu_\Ob- \frac {\rho}{f_0^2}\na f_0$$ is small if
  $\eps$ and thus $\rho(p_0)$ is small, more precisely,
  \beq \label{est:ableit-w} \abs{\bar\na w-\frac1{f_0}\nu_\Ob}\leq
  C\cdot \rho\eeq for a constant $C$ depending only on $L$ and the
  choice of $f_0$.

  In particular, the normal to $\mathcal S_\epsilon$ at $p_0$ is given
  by \[\nu_{\mathcal{S}_\eps}(p_0)= \nu_\Ob(p_0)+ \rho(p_0)\cdot \xi\]
  for some vector $\xi$ whose length is again bounded in terms of the
  function $f_0$ and $L$.
  
  Similarly, we can adjust the orthonormal basis $(e_a)$ of the
  tangent space to $\partial \Ob_\delta$, $\delta=\rho(p_0)=\eps\cdot
  f_0(p_0)$, to give an orthonormal basis $e_a+\rho \cdot \xi_a$ of
  $T_{p_0}\mathcal{S}_\eps$, again with $\abs{\xi_a}\leq C$ as above.
  \par To prove the claim we now show that
  \[\big| A^{\mathcal S_\epsilon}(p_0)\big(e_a+ \rho
  \cdot\xi_a,e_b+\rho\cdot \xi_b\big)\big| \le \big|
  A^{\Ob}(p_0)(e_a,e_b)\big| +C\rho\cdot \left(1+\abs{A^{
      \Ob}(p_0)}\right).\]

  For this we first observe that the final term of
    \begin{align*}
    D^2w
    =&\,\frac1{f_0}D^2 \rho
    -\frac1{f_0^2}\left(D\rho\otimes Df_0+D f_0\otimes D\rho\right) 
    -\frac
    {\rho} {f_0^3}\big( f_0\cdot D^2f_0 -2Df_0\otimes Df_0\big),
  \end{align*}
  which contains $\rho$ itself rather than a derivative of it, must be
  small if $\eps$ is small.

  As $e_a$ is orthogonal to $\nu_\Ob$, we have $D\rho(e_a)=0$, so
  evaluating the second term for the basis $(e_a+\rho\xi_a)$ of
  $T_{p_0}\mathcal{S}_\eps$ gives also just a contribution of order $C
  \rho$, again with $C$ depending only on $f_0$ and $L$, in particular
  independent of the obstacle since $\abs{D\rho}=1$.

  Finally observe that the restriction of $D^2\rho $ to
  $T_{p_0}\partial\O_\delta$ is nothing else than the second
  fundamental form $A^\Ob$ of the level sets of the obstacle while
  $D^2\rho(\nu_\Ob,\cdot)$ vanishes.

  Combined with \eqref{est:ableit-w} we thus find that for $\eps>0$
  sufficiently small \beqas \big| A^{\mathcal
    S_\epsilon}(p_0)\big(e_a+ \rho \xi_a,e_b+\rho \xi_b\big)\big| \leq
  &\, \frac1{\abs{D w}}\left[ \frac{1}{f_0}\abs{A^\Ob(p_0)(e_a,e_b)}+C
    \rho \big(1+\abs{A^{\Ob}(p_0)}\big)\right]
  \\
  \leq &\, \abs{A^\Ob(p_0)(e_a,e_b)}+C\eps (1+\abs{A^{\Ob}(p_0)})
  \eeqas with constants that depend only on $L$ and the function
  $f_0$.  The first claim of the lemma immediately follows.

  To obtain the second claim, we recall the well known fact, see
  e.g. \cite[Lemma 14.17]{GT}, that in a tubular neighbourhood of
  $\partial \Ob$ one can express the principal curvatures of the level
  sets $\partial \Ob_\delta$ in terms of $\delta$ and the principal
  curvatures of $\partial \Ob$.  In particular, there is a constant
  $\eps_2>0$ depending only on $\sup_{\partial \Ob\cap\{x^{n+2}>L-1\}}
  \abs{A^{\Ob}}$ so that for any $p$ with $\abs{\dist_\pO(p)}\leq
  \eps_2$, we have $\abs{A^\Ob(p)}\leq \frac32 \sup_{\partial \Ob\cap
    B_1(p)}\abs{A^\Ob}$.  Reducing $\eps_2$ if necessary and combining
  this with the estimate proven above immediately yields the
  second claim.
\end{proof}

\begin{proof}[Proof of Lemma \ref{lemma:penetration}]
  \neueZeile
  \begin{enumerate}[(i)]
  \item As $\alpha_\epsilon\ge0$, any constant function $u_1$ fulfils
    $\dot u_1\le 1\cdot(0+\alpha_\epsilon)$, i.\,e.{} is a subsolution
    to $\dot u=\sqrt{1+|Du|^2}\cdot\left(\divergenz
      \left(\fracd{Du}{\sqrt{1+|Du|^2}}\right)
      +\alpha_\epsilon\right)$. In particular, the constant $L$ acts
    as a lower barrier for the solution $u$ of \eqref{eq:eps L
      graph}. 
  \item We choose a monotonically nonincreasing function $f_0\in
    C^2_{\loc}(\R,\R^+)$ such that
    \[f_0(s)\ge-\beta^{-1}\left(2\sqrt{n}\cdot\big[
      \sup\limits_{\{x^{n+2}\ge
        s-1\}\cap\partial\O}\left|A^\O\right|+1\big]\right)\] and
    consider as comparison surface $\mathcal S_\epsilon$ for
    $\epsilon\in(0,\epsilon_2)$ as in Lemma
    \ref{lemma:comp-surf-zwei-ff}.  Given an arbitrary point
    $p\in\mathcal S_\epsilon \cap\left\{x^{n+2}\ge L-1\right\}$, we
    observe that \beqas \left|H^{\mathcal S_\epsilon}(p)\right| \le &\,
    \sqrt{n}\left|A^{\mathcal S_\epsilon}(p)\right|\le
    2\sqrt{n}\cdot\big[\sup\limits_{\partial\O\cap\left\{x^{n+2}\ge
        p^{n+2}-1\right\}} \left|A^{\O}\right| +1\big]\\
     \le &\,\beta\left(-f_0\left(p^{n+2}\right)\right)
    =\beta\big(\tfrac{\dist_{\pO}(p)}{\eps}\big)=\aleps(p).  \eeqas
    Consequently, the stationary hypersurface $\mathcal
    S_\epsilon\cap\left\{x^{n+2}>L-1\right\}$ is a subsolution to
    \eqref{eps fluss}.
  \item The maximum of two subsolutions is again a subsolution, for
    example in the viscosity sense. Therefore $\graph u$ remains above
    both $\mathcal S_\epsilon$ and the plane
    $\left\{x^{n+2}=L\right\}$ for all times and \eqref{est:pen-depth}
    is valid with $C_0(\ell)=f_0(\ell-1)$.  \qedhere
  \end{enumerate}
\end{proof}

Based on Lemma \ref{lemma:penetration}, we will henceforth assume

\begin{assumption}[Standard assumption on $\eps$]
  \label{assumption:epsilon-standard}
  Given a number $L\in \R$ and an initial surface $M_0$ (disjoint from
  the obstacle) contained in $\{x^{n+2}\geq L\}$, we consider the
  evolution equation \eqref{eps fluss} only for values of $\eps\in
  (0,\eps_0(L))$, the number given by Lemma \ref{lemma:penetration}.
\end{assumption}

As a consequence of Lemma \ref{lemma:penetration} and its proof, we
get the following more general bounds on the penetration depth of
solutions to \eqref{eps fluss}

\begin{corollary}
  \label{cor:K-introd}
  Let $\O$ be an obstacle as in Definition \ref{def: ini data} which
  we furthermore assume to be of class $C^2$ and let $\ell>-\infty$ be
  any number. Then there exist $K<\infty$ and $C_0>0$ such that the
  following holds true.
  \par
  Let $(M_t)_t$ be a smooth solution of \eqref{eps fluss} (with $\eps$
  satisfying the standard assumption) which is initially disjoint from
  the obstacle. Then $\dist_{\partial\O}(p)\geq -C_0\eps$ and
  $$\aleps(p)+\abs{A^\Ob}(p)+v^\Ob(p)\leq K$$
  for any $p\in M_t\cap \Ob\cap \{x^{n+2}\geq \ell\}$ and any $t\ge0$.
\end{corollary}

We remark that the above constant $K$ depends only on local
$C^2$-bounds of the obstacle.  In particular, while in Definition
\ref{def: ini data} the assumed regularity of the obstacle is only
$C^{1,1}$ and not $C^2$, we can and will approximate such obstacles by
smooth obstacles with locally bounded $C^{2}$-norm, so Corollary
\ref{cor:K-introd} will still apply with constants depending only on
the local $C^{1,1}$-norms of the original obstacle $\Ob$.
\par
In the following sections, we shall derive a priori estimates for
solutions of \eqref{eps fluss} in such halfspaces $\left\{x^{n+2}\geq
  \ell\right\}$ and for this we shall often use

\begin{assumption}[Assumptions for a priori estimates in $\{x^{n+2}\ge
  \ell\}$] \label{ass:a-priori-h}\neueZeilealt
  We consider solutions $(M_t)_t$ of \eqref{eps fluss} with the
  following properties: For some $a>0$
  \begin{enumerate}[(i)]
  \item each $M_t\cap\{x^{n+2}>\ell-a\}$, $t\ge0$, is a graphical, 
    smooth submanifold without boundary and
  \item each $M_t\cap \{x^{n+2}\ge\ell-a\}$ is compact.
  \end{enumerate}
\end{assumption}

\section{$C^1$-estimates for the graphical flow: gradient
  function} \label{section:C1} 

We combine the evolution equation for
the gradient function given in Lemma \ref{v evol} with the key
observation concerning the additional term $\la\naM\aleps,e_{n+2}\ra$
made in Remark \ref{rem:sign nabla alpha eta} and a localisation
argument to prove

\begin{proposition}
  \label{prop:v-est}
  Let $\ell\in\R$ and let $(M_t)_t$ be a solution of \eqref{eps
    fluss}, with $\epsilon$ as in Assumption
  \ref{assumption:epsilon-standard}, such that Assumption
  \ref{ass:a-priori-h} is satisfied.  Then the gradient function is
  controlled by
  \[(U-\ell)^2\cdot v\le \sup\limits_{M_0\cap\left\{x^{n+2}\ge
      \ell\right\}} (U-\ell)^2\cdot v +C(\ell),\] for all times and in
  all points with height $U\geq \ell$.  Here $C(\ell)$ depends only on
  $\max\limits_{M_0}U-\ell$ and the bounds for $v_{\O}$ and
  $\alpha_\epsilon$ from Corollary \ref{cor:K-introd}.
\end{proposition}
\begin{proof}
 We may assume without loss of generality that $\ell=0$. 
  We want to apply the maximum principle to the function
  \begin{align*}
    w:=&\,U^2v\umbruch \intertext{and obtain by direct computation}
    \dot w-\Delta w=&\,2Uv\big(\dot U-\Delta U\big) +U^2\left(\dot
      v-\Delta v\right) -2v|\nabla U|^2 -4U\langle\nabla U,\nabla
    v\rangle\umbruch\\
    =&\,2Uv\frac{\alpha_\epsilon}v +U^2\left(-|A|^2v-\tfrac2v|\nabla
      v|^2 +v^2\left\langle\naM\alpha_\epsilon,
        {e_{n+2}}\right\rangle\right)\\ &\,-2v|\nabla U|^2
    -4U\langle\nabla U,\nabla v\rangle.
  \end{align*}
  At a spatial maximum of $w$, we obtain
  \begin{align*}
    0=&\,2Uv\nabla U+U^2\nabla v,\umbruch\\
    \dot w-\Delta w=&\,2U\alpha_\epsilon -U^2|A|^2v-\frac2v|\nabla
    v|^2U^2
    +v^2U^2\left\langle\naM\alpha_\epsilon,{e_{n+2}}\right\rangle\\
    &\,-2v|\nabla U|^2 +2U^2\frac1v|\nabla v|^2\umbruch\\
    \le&\,2U\alpha_\epsilon
    +v^2U^2\left\langle\naM\alpha_\epsilon,{e_{n+2}}\right\rangle
    -2v\left(1-\tfrac1{v^2}\right), \intertext{where we have used,
      setting $\eta=e_{n+2}$ and observing $|\eta|=1$, that} |\nabla
    U|^2=&\,\eta_\gamma\nabla_iF^\gamma
    g^{ij}\nabla_jF^\zeta\eta_\zeta
    =\eta_\gamma\left(\delta^{\gamma\zeta}
      -\nu^\gamma\nu^\zeta\right)\eta_\zeta
    =|\eta|^2-\langle\nu,\eta\rangle^2 =1-\tfrac1{v^2}.
  \end{align*}
  If $w$ is large, $v$ is also large since the hyperplane
  $\left\{x^{n+2}=\sup u_0\right\}$ is a stationary solution of the
  flow and hence acts as an upper barrier. In this situation,
  $\left\langle\naM\alpha_\epsilon, {e_{n+2}}\right\rangle\le0$
  according to Remark \ref{rem:sign nabla alpha eta}. The term
  $2U\alpha_\epsilon$ is uniformly bounded and can be absorbed as
  $-2v+2/v\le-v$ for $v\ge2$. Hence the claimed inequality follows
  from the maximum principle as $w$ vanishes at height $\ell$.
\end{proof}

\section{Controlling the second fundamental form}\label{section:A}

In this section we analyse the evolution of the second fundamental
form under the flow \eqref{eps fluss}. According to \eqref{eq:evolA2},
we have 
\begin{align}
  (\tdt-\Delta)\abs{A}^2=&\,-\abs{\na A}^2+2\abs{A}^4-2\aleps
  A^k_iA^i_jA^j_k\nonumber\\
  \label{eq:evolA3}
  &\,-2\bepspp\la \nu_\Ob,\pxi F\ra \cdot \la \nu_\Ob,\pxj F\ra A^{ij}
  -2\bepsp\tilde A_{ij}^\Ob A^{ij} +2\bepsp \la\nu,\nu_\Ob\ra
  \abs{A}^2,
\end{align}
where the first two terms agree with the evolution equation for
standard mean curvature flow.

The additional terms are all supported on the obstacle though with
vastly different behaviour as $\eps\searrow 0$, depending on whether or
not the term contains derivatives of the penalty function $\aleps$.

Namely, as $\aleps$ is bounded uniformly in time in every halfspace
$\left\{x^{n+2}\geq\ell\right\}$, see Section \ref{sect:penetration},
the term $-2\aleps A^k_iA^i_jA^j_k$ will be dominated by $ \abs{A}^4$
in points where the second fundamental form is large and as such will
not play an important role, no matter how small $\eps$ is.

Conversely, all other terms contain derivatives of $\aleps$ and can
thus be of order $\eps^{-1}$ (for first order derivatives as occurring
in the last two terms in \eqref{eq:evolA3}) or even $\eps^{-2}$ (for
the other additional term) in points of the obstacle that can a priori
be reached by the evolving surface, compare Section
\ref{sect:penetration}.

These terms cannot be expected to have a sign so that we need to
construct a modified second fundamental form quantity in order to be
able to apply the maximum principle. 

This construction is done in three steps, first replacing $\abs{A}^2$
with a quantity $f$ whose evolution equation resembles more closely
the one of $\abs{A}^2$ for standard mean curvature flow, then,
similarly to \cite{EckerHuiskenInvent} further modifying this to
obtain a quantity $G$ for which $(\dt-\Delta)G$ is negative for large
values of $G$ and controlled gradient and then finally by localising
in space-time. 

We first prove 
\begin{lemma}\label{lemma:f}
  For any $\eta\in(0,1)$ and $\ell>-\infty$, there exists a constant
  $\gamma_0\in(0,1]$, so that to any $\gamma\in (0,\gamma_0]$, we can
  choose $1\le\bar F=\bar F(\eta,\ell,\gamma)<\infty$, such that the
  following holds true.\par
  Let $(M_t)_t$ be a smooth solution of the flow \eqref{eps fluss}
  (for $\eps$ in the range $(0,\eps_1)$ as discussed in Assumption
  \ref{assumption:epsilon-standard}). Then the inequality
  \beqa \label{est:lemma-f}
  e^{-\gamma\aleps}\left(\tdt-\Delta_M\right)
  \left(e^{\gamma\aleps}\abs{A}^2\right)\leq&\,
  -(2-\eta)\abs{\na A}^2+ (2+\eta)\abs{A}^4\\
  &\,-\abs{\bepsp}\cdot \la \nu,\nu_\Ob\ra_+\abs{A}^2-\frac\gamma4
  \bepspp\abs{A}^2\abs{P_{TM}\nu_\Ob}^2 \eeqa holds in every point
  $p\in M_t\cap \{x^{n+2}\geq \ell\}$ in which
  $$\abs{A}\geq \bar F.$$
\end{lemma}

Recall that $\alpha_\epsilon$ is uniformly bounded in points $p\in
M_t\cap\{x^{n+2}\ge \ell\}$, see Corollary~\ref{cor:K-introd}.  Hence
in points where $e^{\gamma\alpha_\epsilon}|A|^2$ is large, $|A|$ is
also large and the estimate above applies. Therefore inequalities as
in Lemma \ref{lemma:f}, valid only where $|A|$ is large and thus of a
much simpler form than the general evolution equation, are suitable to
derive upper bounds on the second fundamental form.

We remark that while the present lemma makes no use of the
$C^1$-bounds on the evolving hypersurface derived earlier, such bounds
will be crucial in the following lemma.

\begin{lemma}\label{lemma:mod-f}
  For any numbers $M<\infty$ and $\ell>-\infty$, there exist numbers
  $\gamma,k>0$ as well as $\bar F<\infty$, such that the following
  holds. Let $(M_t)_t$ be a smooth solution of \eqref{eps fluss} for
  some $\epsilon\in(0,\epsilon_1)$ as in Assumption
  \ref{assumption:epsilon-standard} and set
  $$G:=h\left(v^2\right)\cdot e^{\gamma\aleps}\cdot
  \abs{A}^2,\quad\text{where}\quad h(y)=y\cdot e^{ky}.$$ Then
  \begin{align*} \left(\dt-\Delta\right) G +\frac1h\langle\nabla
    h,\nabla G\rangle \leq &\,-\frac {k}8 \left[ h
      e^{\gamma\alpha_\epsilon}|\nabla A|^2
      +G|A|^2+ G |\nabla v|^2\right] \\
    &\,-\left[\frac\gamma8\beta_\epsilon''|P_{TM}\nu_{\O}|^2
      +\frac12|\beta_\epsilon'|\langle\nu,\nu_{\O}\rangle_+\right]
    \cdot G
  \end{align*}
  holds in every point $p\in
  M_t\cap\left\{x^{n+2}\geq \ell\right\}$, where $\abs{A}$ is large
  and the gradient function $v$ of $M_t$ is bounded, namely
  $$\abs{A(p)}\geq \bar F, \text{ while } v(p)\leq M.$$ 
\end{lemma}

\begin{proof}[Proof of Lemma \ref{lemma:f}]
  Let $\eta>0$ and $\ell>-\infty$ be given. Let $K$ be as in Corollary
  \ref{cor:K-introd} and let $(M_t)_t$ be a solution of the flow
  \eqref{eps fluss} for some number $\eps\in (0,\eps_1)$ as in
  Assumption \ref{assumption:epsilon-standard}.  Then for $\gamma$ in
  a range $(0,\gamma_0)$ to be determined later, we set
  $$f=f_\gamma=e^{\gamma \aleps}\abs{A}^2$$ and compute, using
  \eqref{eq:evolA2} and \eqref{eq:evol-aleps}, 
  \begin{align*} 
    e^{-\gamma \aleps}\big[ \left(\tdt-\Delta\right)f\big]
    =&\,\left(\tdt-\Delta\right)\abs{A}^2-2\gamma\big\la \naM\aleps,
    \na
    \abs{A}^2\big\ra\\
    &\,+\gamma\abs{A}^2\cdot \left(\tdt-\Delta\right)\aleps
    -\gamma^2\abs{\naM\aleps}^2\abs{A}^2\\
    =&\,-2\abs{\na A}^2+2\abs{A}^4-2\aleps A^k_iA^i_jA^j_k\\
    &\,-\bepspp\la \nu_\Ob,\pxi F\ra \cdot \la \nu_\Ob,\pxj F\ra
    A^{ij}
    -2\bepsp\tilde A_{ij}^\Ob A^{ij}\\
    &\,+2\bepsp \la\nu,\nu_\Ob\ra \abs{A}^2
    -2\gamma\bepsp \big\la P_{TM}\nu_\Ob,\na \abs{A}^2\big\ra\\
    &\,+\gamma\big[\bepsp\la
    \nu_\Ob,\aleps\nu\ra-\bepspp\abs{P_{TM}\nu_\Ob}^2-\bepsp\tilde
    A_{ij}^\Ob g^{ij}\big] \abs{A}^2\\
    &\,-\gamma^2\abs{\bepsp}^2\abs{P_{TM}\nu_\Ob}^2\abs{A}^2.  
  \end{align*}  
  Dropping the last, obviously non-positive term and using Young's
  inequality as well as Kato's inequality $\abs{\na \abs{A}}\leq
  \abs{\na A}$, we obtain \beqa \label{est:f0}
  e^{-\gamma\aleps}\left[\left(\tdt-\Delta\right)f\right] \leq&\,
  -(2-\eta)\abs{\na A}^2+\left(2+\frac{\eta}{2}\right) \abs{A}^4
  +\frac2{\eta}\abs{A}^2 \aleps^2
  \\
  &\,+C\abs{\bepsp}\cdot \abs{A^\Ob} \cdot \big(\gamma \abs{A}^2
  +\abs{A}\big)\\
  &\,-\abs{P_{TM}\nu_\Ob}^2\cdot \left[ \bepspp\big(\gamma
    \abs{A}^2-\abs{A}\big)
    -\frac{4\gamma^2}{\eta}\abs{\bepsp}^2\abs{A}^2\right]\\
  &\,-\la\nu,\nu_\Ob\ra (2+\gamma\aleps)\abs{\bepsp}\abs{A}^2.\\
  \eeqa
  
  To rewrite this expression in the form
  \beqa \label{est:f1}
  e^{-\gamma\aleps}\big[\left(\tdt-\Delta\right)f\big] \leq
  &\, -(2-\eta)\abs{\na
    A}^2+\left(2+\frac{\eta}{2}\right) \abs{A}^4
  +\frac2{\eta}\abs{A}^2 \aleps^2
  \\
  &\,-\abs{P_{TM}\nu_\Ob}^2\cdot T_1 -\la \nu,\nu_\Ob\ra \cdot T_2,
  \eeqa 
  we then use that
$$1=\abs{\nu_\Ob}^2=\abs{P_{TM}\nu_\Ob}^2+\la \nu,\nu_\Ob\ra^2$$
to split the term on the second line of \eqref{est:f0} into suitable
multiples of $\abs{P_{TM}\nu_\Ob}^2$ and of $ \la \nu,\nu_\Ob\ra$ and
find that \eqref{est:f1} is valid for 
\begin{align*}
  T_1:=&\,\left[\gamma
    \bepspp-\frac{4\gamma^2}\eta\abs{\bepsp}^2-C\gamma\abs{\bepsp}
    \abs{A^\Ob}\right]\cdot \abs{A}^2
  -\left[C\abs{\bepsp}\abs{A^\Ob}+\bepspp \right]\cdot \abs{A}\\
  \geq&\,\gamma\cdot\left[\bepspp
    -\gamma(4\eta^{-1}+1)\abs{\bepsp}^2\right]
  \abs{A}^2-\left[\bepspp+\abs{\bepsp}^2\right] \cdot \abs{A}\\ 
  &\,-C\left\lvert A^\Ob\right\rvert^2(\abs{A}^2+1)
  \intertext{and} 
  T_2=&\,(2+\gamma\aleps)\abs{\bepsp}\abs{A}^2- C\la
  \nu,\nu_\Ob\ra \abs{\bepsp}\cdot \abs{ A^\Ob} \cdot \left(\gamma
    \abs{A}^2 +\abs{A}\right),
\end{align*}
 $C=C(n)$ some universal
constants.

We will first show that the dominating term in $T_1$ is given by
$\gamma \bepspp\abs{A}^2>0$, so that we obtain a negative contribution
to the right-hand side of \eqref{est:f1} scaling as $\epsilon^{-2}$ in
points where $P_{TM}\nu_\Ob$ is non-zero, i.\,e.{} in points where the
tangent plane of the evolving hypersurface and the obstacle do
\textit{not} coincide.

Conversely, as both the obstacle and the evolving hypersurface are
graphical, it is precisely in points where the two tangent planes
coincide that $\la \nu,\nu_\Ob\ra$ is maximal, i.e. equal to one, so,
as we shall see, we again get a large negative contribution to the
right hand side of \eqref{est:f1} now coming from the dominating term
$2\abs{\bepsp}\abs{A}^2$ of $T_2$.

To begin with we show\\
\textit{Claim:} Given any $\eta>0$ there exists
$\gamma_0>0$ such that for any $\gamma\in(0,\gamma_0)$ there is a
number $\bar F$ such that
$$T_1\geq \left(\frac{\gamma}2\bepspp-C\cdot K^2\right)\cdot \abs{A}^2$$
in every point $p\in M_t\cap \left\{x^{n+2}\geq\ell\right\}$ in which
$\abs{A}\geq \bar F$. Here $C$ is a universal constant and $K=K(\ell)$
is the number given in Corollary \ref{cor:K-introd}.

To prove this claim, we first recall from Corollary \ref{cor:K-introd}
that $\dist_{\pO}(p)\geq -c_0\cdot \eps$, $c_0=c_0(\ell)$.  Thus
$\beps$ and its derivatives need to be evaluated only for arguments
contained in an interval $[-c_0 \eps, \infty)$ where \beq
\frac{(\bepsp)^2}{\bepspp}\leq
\sup_{[-c_0,0]}\frac{(\beta')^2}{\beta''}\leq
C_1\label{est:der-beta}\eeq is bounded by a constant depending only on
$c_0$ (and thus $\ell$) and the function $\beta$, which we had chosen
so that $\beta'''\le0$.

In points where $\abs{A}$ is large, $\abs{A}\geq \bar F$ for $\bar
F\geq 1$ still to be determined, we thus get
\beqa \label{est:T1-hilfs} T_1 &\geq \gamma\bepspp\cdot \bigg[1-\gamma
C_{1}\left(4\eta^{-1}+1\right)-\left(\gamma\bar
  F\right)^{-1}(1+C_1)\bigg]\abs{A}^{2} -C\abs{A^\Ob}^2\abs{A}^2.
\eeqa Choosing $\gamma_0\in(0,1)$ small enough so that $\gamma_0
C_1(4\eta^{-1}+1)\leq \frac14$, and then, for each
$\gamma\in(0,\gamma_0)$, selecting a number $\bar F$ large enough so
that $\left(\gamma\bar F\right)^{-1} (1+C_1)\le\frac14$, we thus find
as claimed that \beq \label{est:T1} T_1\geq \frac{\gamma}2 \bepspp
\abs{A}^2-C\cdot K^2\abs{A}^2 \eeq where we use Corollary
\ref{cor:K-introd} to deal with the last term in \eqref{est:T1-hilfs}.

To analyse $T_2$, we first observe that 
\begin{align*}
  \left|T_2-2\left|\beta_\epsilon'\right||A|^2\right|\le&\,
  \gamma\alpha_\epsilon\left|\beta_\epsilon'\right||A|^2
  +C\left|\beta_\epsilon'\right|\cdot\left|A^\O\right|
  \cdot|A|^2\cdot\left(\gamma+|A|^{-1}\right)\umbruch\\
  \le&\,\gamma_0 K \left|\beta_\epsilon'\right||A|^2 +C\gamma_0 K
  \left|\beta_\epsilon'\right||A|^2 +C\bar
  F^{-1}K\left|\beta_\epsilon'\right||A|^2\umbruch\\
  \le&\,CK\left(\gamma_0+\bar
    F^{-1}\right)\left|\beta_\epsilon'\right||A|^2,
\end{align*}
since we only need to consider points with $|A|\ge\bar F$. 
After possibly reducing $\gamma_0$ and increasing $\bar F$, we thus
obtain \beq
3\abs{\bepsp}\abs{A}^2\geq T_2\geq
\abs{\bepsp}\abs{A}^2.\label{est:T2}\eeq Remark that these expressions
only scale as $\eps^{-1}$ and not as $\eps^{-2}$ like the leading
order term of $T_1$.

This difference is crucial since we cannot expect to control the sign
of $\la \nu,\nu_\Ob\ra$ and will thus need to rely on the contribution
of $T_1$ to \eqref{est:f1} in points where this inner product is
negative.  While not necessarily positive, we observe that since both
the obstacle and the evolving hypersurface are graphical, this inner
product is bounded away from $-1$. Namely writing $\nu_\Ob=
\la\nu_\Ob,e_{n+2}\ra e_{n+2}+P_{\R^{n+1}}\nu_\Ob$, where
$P_{\R^{n+1}}\nu_\Ob$ is the orthogonal projection of $\nu_\Ob$ onto
$\R^{n+1}\times \{0\}$, we find 
\begin{align*} \la \nu_\Ob,\nu\ra=&\,\la \la \nu_\Ob,e_{n+2}\ra
  e_{n+2}+P_{\R^{n+1}}\nu_\Ob, \nu\ra=\la \nu_\Ob,e_{n+2}\ra \cdot
  \la\nu,e_{n+2}\ra +\la P_{\R^{n+1}}\nu_\Ob,\nu\ra \\
  \geq&\,(v\cdot v_{\Ob})^{-1} -\abs{P_{\R^{n+1}}\nu_\Ob}\geq
  -\abs{P_{\R^{n+1}}\nu_\Ob}
  =-\sqrt{1-\la\nu_\Ob,e_{n+2}\ra^2}\\
  \geq&\, -\sqrt{1-K^{-2}},
\end{align*}
with the last inequality due to Corollary \ref{cor:K-introd}.

In points where $\la \nu,\nu_\Ob\ra<0$, we may thus bound 
$$\abs{P_{TM} \nu_\Ob}^2=1-\abs{\la\nu,\nu_\Ob\ra}^2 \geq K^{-2},$$
which in turn gives \beq \label{est:inner-prod-nu} \la\nu_\Ob,\nu\ra
=\la\nu_\Ob,\nu\ra_+-\la\nu_\Ob,\nu\ra_-\geq \la\nu_\Ob,\nu\ra_+
-K^2\abs{P_{TM}\nu_\Ob}^2.\eeq

Considering points $p\in M_t\cap \left\{x^{n+2}\geq \ell\right\}$ with
$\abs{A}\geq \bar F$, we can thus conclude from \eqref{est:T1} and
\eqref{est:T2} that the estimate \beqa \label{est:T1-T2}
\abs{P_{TM}\nu_\Ob}^2T_1+\la \nu,\nu_\O\ra T_2 \geq
|P_{TM}\nu_{\O}|^2\cdot\left(\frac\gamma4\bepspp-C\right)\abs{A}^2+\la
\nu,\nu_\Ob\ra_+\abs{\bepsp}\abs{A}^2 \eeqa holds with a
constant $C=C(K)$, at least if $\la \nu,\nu_\Ob\ra\geq 0$. On the
other hand, if $\la \nu,\nu_\Ob\ra< 0$, we can combine \eqref{est:T1}
and \eqref{est:T2} with \eqref{est:inner-prod-nu} to conclude that  
\begin{align*}
  \abs{P_{TM}\nu_\Ob}^2T_1+\la \nu,\nu_\O,\ra T_2 \geq&\,
  \abs{P_{TM}\nu_\Ob}^2 \cdot
  \left(\frac{\gamma}{2}\bepspp-3\abs{\bepsp}K^2
  \right)\cdot \abs{A}^2-C\cdot\abs{A}^2\\
  \geq&\, |P_{TM}\nu_{\O}|^2\cdot
  \left(\frac\gamma2\bepspp-\gamma^2\abs{\bepsp}^2\right)\abs{A}^2
  -C\left(1+\gamma^{-2}\right)\abs{A}^2\\
  \geq&\, |P_{TM}\nu_{\O}|^2\cdot
  \left(\frac\gamma2\bepspp-\gamma^2C_1\bepspp\right)
  \abs{A}^2-C(\gamma,K)\abs{A}^2\\
  \geq&\, |P_{TM}\nu_{\O}|^2\cdot\left[
    \frac{\gamma}{4}\bepspp-C(\gamma,K)\right] \abs{A}^2
\end{align*}
since $\gamma_0 C_1\leq \frac14$. But in this second case
$\la\nu,\nu_\Ob\ra_+$ is zero which means that \eqref{est:T1-T2} also
holds though now with a constant $C=C(\gamma,K)$. Inserting
\eqref{est:T1-T2} into \eqref{est:f1} thus gives
  \begin{align*} 
    e^{-\gamma\aleps}\left(\tdt-\Delta_M\right)
    \big(e^{\gamma\aleps}\abs{A}^2\big) \leq&\,-(2-\eta)\abs{\na A}^2+
    \left(2+\frac\eta2\right)\abs{A}^4
    +C_2\cdot\abs{A}^2\\
    &\,-\abs{\bepsp}\cdot \la \nu,\nu_\Ob\ra_+\abs{A}^2 -\frac\gamma4
    \bepspp\abs{A}^2\abs{P_{TM}\nu_\Ob}^2
  \end{align*}
  for a constant $C_2$ depending on $\eta$, $\gamma$ as well as
  $K$. Possibly further increasing $\bar F$ (which is allowed to
  depend on all these quantities), we can however assume that $C_2\leq
  \frac\eta2 (\bar F)^2$, so that we can estimate the final term on
  the first line by $\frac\eta 2 \abs{A}^4$ in the points under
  consideration, thus obtaining the claim of the lemma.
\end{proof} 

\begin{proof} [Proof of Lemma \ref{lemma:mod-f}]
  Given a number $M$ and a level $\ell>-\infty$, we let $K$ be as in
  Corollary \ref{cor:K-introd} and consider a smooth solution
  $(M_t)_t$ of the flow \eqref{eps fluss}, $\eps\in (0,\eps_1)$ as in
  Assumption \ref{assumption:epsilon-standard}, in points where $M\ge
  v$. For a number $\eta=\eta(M,\ell)>0$ to be determined below, we
  let $\gamma_0=\gamma_0(\eta,\ell)>0$ be as in Lemma \ref{lemma:f}.

  We then consider the function 
  $$G=h\left(v^2\right)\cdot f,$$
  where $f=e^{\gamma\aleps}\abs{A}^2$ is as in Lemma \ref{lemma:f},
  with $\gamma\in(0,\gamma_0)$ still to be determined, and $h:\R^+\to
  \R^+$ a nondecreasing function which we will later choose as stated
  in the lemma.

  To begin with, we calculate \beqa\label{est:G-0}
  \left(\tdt-\Delta\right)G=&\, h\cdot \left(\tdt-\Delta\right)f+f
  \cdot 2v\cdot h'\cdot \left(\tdt-\Delta\right)v
  \\
  &\,-\left[h'' \cdot \abs{\na\left(v^2\right)}^2+2h'\cdot \abs{\na
      v}^2\right]\cdot f-2 \left\la\na \left(h\left(v^2\right)\right)
    ,\na f\right\ra. \eeqa Here and in the following, $h$ and its
  derivatives are evaluated at $y=v^2$ unless stated otherwise.

  Let now $p\in M_t\cap \left\{x^{n+2}\geq \ell\right\}$ be a point
  where $\abs{A}\geq \bar F$, the number given by Lemma \ref{lemma:f}.
  Inserting the evolution equation \eqref{eq v evol} of the gradient
  function as well as the estimate \eqref{est:lemma-f} into
  \eqref{est:G-0}, we obtain
  \begin{align}
  e^{-\gamma\aleps}\left(\tdt-\Delta\right)G
  \leq&\,
  h\cdot \bigg[-(2-\eta)\abs{\na A}^2+ (2+\eta)\abs{A}^4\nonumber\\
  &\,\qquad -\abs{\bepsp}\cdot \la \nu,\nu_\Ob\ra_+\cdot
  \abs{A}^2-\frac\gamma4
  \bepspp\abs{A}^2\abs{P_{TM}\nu_\Ob}^2\bigg]\nonumber\\ 
  \label{est:g-1} &\,+\abs{A}^2\cdot 2v\cdot h'\cdot \bigg[-\abs{A}^2v
  -\frac2v\abs{\na
    v}^2+v^2\left\la \naM\aleps,e_{n+2}\right\ra\bigg]\\
  &\,-\left(4h''v^2\abs{\na v}^2+2h'\cdot \abs{\na v}^2\right)\cdot
  \abs{A}^2\nonumber\\ &\,- 2e^{-\gamma\aleps}\left\la \na
    \left(h\left(v^2\right)\right),\na f\right\ra .\nonumber  
  \end{align}

We estimate the last term on the third line using Young's inequality as
\begin{align*} 
  \abs{A}^2\cdot 2v\cdot h'\cdot v^2\la \naM\aleps,e_{n+2}\ra
  =&\,2v^3\abs{A}^2h'\bepsp\la P_{TM}\nu_\Ob, e_{n+2}\ra\\
  \leq&\, \gamma^2\frac{(h')^2}{h}\cdot
  v^2\abs{\bepsp}^2\abs{P_{TM}\nu_\Ob}^2|A|^2\\&\,+\gamma^{-2}v^4h\cdot
  \abs{A}^2\\
  \leq&\, \gamma^2\frac{(h')^2}{h}\cdot
  v^2\abs{\bepsp}^2\abs{P_{TM}\nu_\Ob}^2|A|^2\\
  &\,+\gamma^{-2}M^4\cdot e^{-\gamma\aleps}\cdot G.
\end{align*}
Then, as in \cite{EckerHuiskenInvent}, we deal with the last term in
\eqref{est:g-1} by writing one multiple of $e^{-\gamma\aleps}\la \na
(h(v^2)),\na f\ra $ in terms of $G=h\cdot f$ as 
\begin{align*}
  -e^{-\gamma\aleps} \left\la \na\left(h\left(v^2\right)\right), \na
    f\right\ra =&\,-\frac{e^{-\gamma\aleps}} h\cdot \left\la \na
    \left(h\left(v^2\right)\right), \na G\right\ra
  +\frac{e^{-\gamma\aleps}}{h}
  \abs{\na\left(h\left(v^2\right)\right)}^2\cdot f\umbruch\\
  =&\,-\frac{e^{-\gamma\aleps}} h\cdot \left\la \na
    \left(h\left(v^2\right)\right), \na G\right\ra
  +4\frac{(h')^2}{h}v^2\abs{\na v}^2\abs{A}^2
\end{align*}
while rewriting the remaining multiple as
$$e^{-\gamma\aleps}\left\la \na \left(h\left(v^2\right)\right), \na f
\right\ra =\left\la \na \left(h\left(v^2\right)\right),
  \na\left(\abs{A}^2\right)\right\ra+\gamma \abs{A}^2\left\langle \na
  \left(h\left(v^2\right)\right), \naM \aleps\right\rangle $$ and
consequently estimating it, using Kato's and Young's inequality as
well as \eqref{eq:Daleps}, by 
\begin{align}
  \left|e^{-\gamma\aleps}\left\la \na
      \left(h\left(v^2\right)\right),\na f\right\ra\right| \leq&\, 4
  h' v \abs{\na v}\cdot\abs{A}\abs{\na A}+2\gamma h' \cdot v\cdot
  \abs{A}^2\cdot \bepsp\la \na v,P_{TM}\nu_\Ob\ra \nonumber\\
  \leq&\,(2-2\eta)\abs{\na A}^2\cdot h+\frac{4}{2-2\eta}\frac{(h'
    )^2}{h}\abs{\na v}^2v^2\abs{A}^2 \label{est:g-3}\\
  &\,+\eta h\abs{\na v}^2\abs{A}^2
  +\frac{\gamma^2}{\eta}\abs{\bepsp}^2\abs{P_{TM}\nu_\Ob}^2v^2\cdot
  \frac{(h')^2}{h}\cdot\abs{A}^2.\nonumber
\end{align}

Combining \eqref{est:g-1}-\eqref{est:g-3} we thus find that
\begin{align*} 
  \left(\tdt-\Delta\right)G\leq&\, -T_3\left(v^2\right)\cdot
  e^{\gamma\aleps}\abs{A}^4-T_4\left(v^2\right)\cdot
  e^{\gamma\aleps}\abs{A}^2\abs{\na v}^2-T_5^{(\eps)}\left(v^2\right)
  \cdot G\\
  &\,-\eta e^{\gamma\aleps} h\cdot \abs{\na A}^2
  -\frac1{h(v^2)}\left\la\na\left(h\left(v^2\right) \right), \na G
  \right\ra+M^4\gamma^{-2}G,
\end{align*}
where 
\begin{align}\label{eq:T3-T5}
  T_3(y)&:=2h'(y)\cdot y-(2+\eta)h(y),\\
  T_4(y)&:=4h''(y)\cdot y-\left(4+\frac{2}{1-\eta}\right)\cdot
  \frac{y\cdot (h'(y))^2}{h(y)}+6 h'(y)-\eta\cdot h,\nonumber
\end{align}
and
\[T_5^{(\eps)} =\abs{\bepsp}\la \nu,\nu_\Ob\ra_+ +
\abs{P_{TM}\nu_\Ob}^2
\left(\frac\gamma4\bepspp-\gamma^2(1+\eta^{-1})\abs{\bepsp}^2\cdot
  \frac{\left(h'\left(v^2\right)\right)^2 v^2
  }{h^2\left(v^2\right)}\right)\] need all be evaluated at $y=v^2$,
and thus, by assumption, for arguments in the interval $[1,M^2]$.

We will show that all the above terms are strictly positive for
$h(y)=y\cdot e^{ky}$ provided $k$, $\eta$ and $\gamma$ are chosen
suitably (depending on the given numbers $M$ and $ \ell$).

We choose $k:=\left(24M^2\right)^{-1}$ and consider the function
$h(y)=y\cdot e^{ky}$, whose derivatives are given by
$$h'(y)=h(y)\cdot 
\big(\tfrac{1}{y}+k\big),\qquad h''(y)=h(y)\cdot
\big(\tfrac{2k}y+k^2\big).$$ Now selecting $\eta$ as
$\eta=\frac{k}{2}$ we obtain that the first term in \eqref{eq:T3-T5}
is positive, namely \[T_3(y)=(2k y-\eta)h(y)\geq \tfrac32kh(y)\] for
any $y\in\left[1,M^2\right]$ which we recall is the range of $v^2$ for
the points we consider.

Furthermore, as $\frac{2}{1-\eta}=2(1+\lambda \eta)$ for
$\lambda=\frac{1}{1-\eta}\leq \frac{48}{47}$, we can bound 
\begin{align*} 
  T_4(y)=&\,\frac{h(y)}{y}\cdot \bigg[
  4(2k+k^2y)y-(6+2\lambda \eta) (1+ky)^2+6(1+ky)-\eta y\bigg]\\
  =&\, \frac{h(y)}{y}\cdot \bigg[
  2ky-\eta(2\lambda+y)+4k^2y^2-4\lambda\eta
  ky-(6+2\lambda\eta)k^2y^2\bigg]\\
  =&\,\frac{h(y)}y\left[2ky-\frac k2y-k\lambda +4k^2y^2-2\lambda k^2y
    -(6+\lambda k)k^2y^2\right]\\
  \geq&\, \frac{h(y)}{y}\cdot \bigg[ \frac38ky-6k^2y^2\bigg]\geq
  \frac18 kh(y). 
\end{align*}

We recall that so far we have only imposed an upper bound on $\gamma$,
namely $\gamma\in (0,\gamma_0)$, $\gamma_0=\gamma_0(\eta,\ell)$ the
number given by Lemma \ref{lemma:f}. We shall now prove that for
$\gamma$ chosen small enough (depending on $\eta$ and $k$) also
$T_5^{(\eps)}$ will be positive.

Namely, as $\frac{(h'(y))^2\cdot y}{h^2(y)}\le \frac{(h'(y))^2\cdot
  y}{h(y)}=e^{ky}(1+ky)^2\leq 2$ for $y\in \big[1,M^2\big]$, we select
$\gamma\in(0,\gamma_0)$ small enough to assure that
$$2\gamma\left(1+\eta^{-1}\right)\cdot C_1\leq \frac18,$$
$C_1=C_1(\ell)$ as in \eqref{est:der-beta}, in order to get
$$T_5^{\eps}\geq \abs{\bepsp}\la \nu,\nu_\Ob\ra_+ 
+\frac\gamma8 \abs{P_{TM}\nu_\Ob}^2\bepspp.$$

All in all we thus conclude that for points $p$ with $v(p)\leq M$ and
$\abs{A(p)}\geq \bar F$ \begin{align*}\left(\tdt-\Delta\right) G
+\frac1h\langle\nabla h,\nabla G\rangle \leq &\,-\frac{k}{2}h\cdot
e^{\gamma\aleps} |\nabla A|^2
-\frac{3k}2|A|^2\cdot G+\gamma^{-2}M^4\cdot G\\
&\,-\frac{k}{8} |\nabla v|^2 \cdot G
-\left[\frac\gamma8\beta_\epsilon''|P_{TM}\nu_{\O}|^2
  +\frac12|\beta_\epsilon'|\langle\nu,\nu_{\O}\rangle_+\right] \cdot
G.  \end{align*} This implies the claim of the lemma as we may further
increase the number $\bar F=\bar F_\gamma$ determined originally in
Lemma \ref{lemma:f} in order to achieve that $\gamma^{-2}M^4\leq
\frac{k}2\bar F^{2}$, allowing us to absorb the third term
into the second term on the right-hand side.
\end{proof}

We now localise these estimates to be able to apply the maximum
principle in halfspaces.

\begin{proposition}
  \label{prop:C2-est}
  Given any level $\ell\in\R$ and any numbers $M\ge1$ and $Q<\infty$,
  there exists a constant $C$ depending only on $\ell,\,M$, $Q$ and
  the obstacle such that for solutions $(M_t)_t$ of \eqref{eps fluss}
  evolving from an initial surface $M_0=graph(u_0)$ disjoint from the
  obstacle and with $\sup u_0\leq Q$ that satisfy \[v\le
  M\quad\text{on}\quad
  M_t\cap\left\{x^{n+2}\ge\ell-1\right\}\quad\text{for every }t\ge0,\]
  the second fundamental form is controlled on $\Mth$ by
  \begin{enumerate}[(i)]
  \item \[(U-\ell)^4\cdot|A|^2\le \frac Ct\] for $t\in(0,1]$
    and 
  \item \[(U-\ell)^4\cdot|A|^2\le
    C\cdot\left(1+\sup\limits_{\Mnullth}(U-\ell)^4\cdot |A|^2\right).\]  
    for all $t\geq 0$.
  \end{enumerate}
\end{proposition}
This proposition is an immediate corollary of the subsequent Lemma
\ref{tU4g lambda U4v2 lem} and the maximum principle.

\begin{lemma}
  \label{tU4g lambda U4v2 lem}
  Let $\ell,\,M,Q\in\R$ and $(M_t)_t$ be as in Proposition
  \ref{prop:C2-est}. Let $G$ be the second fundamental form quantity 
  considered in Lemma
  \ref{lemma:mod-f}. Define 
  \begin{align*}
    w_0=&\,(U-\ell)^4\cdot G&\text{for all $t\ge0$}
    \intertext{and} 
    w_1=&\,t(U-\ell)^4\cdot G+\lambda (U-\ell)^4v^2&\text{for
      $t\in[0,1]$.}
  \end{align*}
  Then there exists a constant $D$ such that
  \[\left(\tdt-\Delta\right) w_i\le0,\quad i=0,\,1,\]
  in every point where the respective function fulfils $w_i\ge D$ and
  $\nabla w_i=0$. 
\end{lemma}
\begin{proof}
  We may assume that $\ell=0$.  Let $\theta\in\{0,1\}$ and set 
  $$w:=\,((1-\theta)+t\theta)U^4G+\theta\lambda U^4v^2,$$
  where $\lambda$ is a large constant that will be fixed later.  This
  allows us to consider the two cases simultaneously. If $\theta=0$,
  we obtain a priori estimates up to $t=0$ provided that $|A|^2$ is
  initially bounded. If $\theta=1$, we obtain local in time a priori
  estimates. \par
  We first observe that since $\{x^{n+2}=Q\}$ lies above $M_0$, it
  must be disjoint from the obstacle and consequently serves as upper
  barrier for $U$ for all times.  In addition to the $C^1$-estimates
  we have furthermore bounds on $\aleps$ thanks to Corollary
  \ref{cor:K-introd}.  Consequently, if $w$ is large, say $w\geq D$,
  then also $\abs{A}^2$ must be large. In particular, for a suitable
  choice of $D$, it is enough to consider points with $\abs{A}^2\geq
  \bar F$, the constant of Lemma \ref{lemma:mod-f}.  We can thus
  estimate, using Lemmas \ref{v evol} and \ref{lemma:mod-f}
  \begin{align}
    \left(\tdt -\Delta\right) w=&\,\theta U^4G
  +4((1-\theta)+t\theta)U^3G\left(\dot U-\Delta U\right)\nonumber\\
  &\,+((1-\theta)+t\theta)U^4(\dot G-\Delta G)
  +4\theta\lambda U^3v^2\left(\dot U-\Delta U\right)\nonumber\\
  &\,+2\theta\lambda U^4v(\dot v-\Delta v)\nonumber\\
  &\,-12((1-\theta)+t\theta)U^2G|\nabla U|^2
  -8((1-\theta)+t\theta)U^3\langle\nabla U,\nabla G\rangle\nonumber\\
  &\,-12\theta\lambda U^2v^2|\nabla U|^2 -16\theta\lambda
  U^3v\langle\nabla U,\nabla v\rangle -2\theta\lambda U^4|\nabla
  v|^2\umbruch\nonumber\\
  \le&\,\theta U^4G
  +4((1-\theta)+t\theta)U^3G\frac{\alpha_\epsilon}v\nonumber\\
  &\,+((1-\theta)+t\theta)U^4\left(\underline{-\frac{k}8G\abs{A}^2}
    -\frac12 |\beta_\epsilon'| \langle\nu,\nu_{\O}\rangle_+\cdot G
    -\frac1h\langle\nabla h,\nabla G\rangle\right)\nonumber\\
  &\,+4\theta\lambda
  U^3v^2\frac{\alpha_\epsilon}v \label{est:w1}\\
  &\,+2\theta\lambda U^4v\left(\underline{-|A|^2v} -\frac2v|\nabla
    v|^2 +v^2\left\langle\na_M\alpha_\epsilon,
      e_{n+2}\right\rangle\right)\nonumber\\
  &\,-12((1-\theta)+t\theta)U^2G|\nabla U|^2
  -8((1-\theta)+t\theta)U^3\langle\nabla U,\nabla G\rangle\nonumber\\
  &\,-12\theta\lambda U^2v^2|\nabla U|^2 -16\theta\lambda
  U^3v\langle\nabla U,\nabla v\rangle -2\theta\lambda U^4|\nabla
  v|^2. \nonumber
\end{align}

Using that $G\leq C\abs{A}^2$, we can use the first underlined term
above to absorb (upto an additive constant $C$) the second term of the
right-hand side. We drop the first term of the penultimate
line. Provided $\lambda$ is chosen sufficiently large, we can
furthermore absorb the first term on the right hand side into the
second underlined term.  Estimating also the penultimate term using
Young's inequality and bounding the first order terms by a constant,
this reduces the above inequality to 
\begin{align} \left(\tdt -\Delta\right)w\leq&\,
  -c_1((1-\theta)+t\theta)U^4G^2-\theta \lambda
  U^4v^2\abs{A}^2-\left(6-\tfrac14\right)\theta\lambda U^4\abs{\na
    v}^2  +C\nonumber\\
  \label{est:w2} &\,+I+II+III_\eps 
\end{align}
for some $c_1>0$ and a constant $C<\infty$ which may also depend on
$\lambda$.
  
  Here $I$ and $II$ stand for the terms appearing on the right hand
  side of \eqref{est:w1} that contain $\na G$
  while $$III_\eps:=2\theta\lambda U^4v^3\la \na
  \aleps,e_{n+2}\ra-((1-\theta)+t\theta)U^4G\frac12\abs{\bepsp}\la\nu,
  \nu_\Ob\ra_+.$$ 
  
  Since we only consider points at which $\na w=0$ we can replace $\na
  G$ in both $I$ and $II$ using
   \begin{align*}
    0=&\,\nabla w\\ =&\,4((1-\theta)+t\theta)U^3G\nabla U
    +((1-\theta)+t\theta)U^4\nabla G +4\theta\lambda U^3v^2\nabla U
    +2\theta\lambda U^4v\nabla v.
  \end{align*}
 Recall furthermore that 
 \[|\nabla v|^2=g^{ij}\na_i v\na_j v=v^4 g^{ij}A_i^k\na_k U A_j^l
 \na_l U\leq v^4\abs{\na U}^2\abs{A}^2 \leq v^4\abs{A}^2\le cG,\]
 compare \eqref{eq:deriv-v} and that $h(y)=ye^{ky}$ so, writing for
 short $\na h$ for $\na\left(h\left(v^2\right)\right)$,
\[\frac{\nabla h}{h}=\frac{h'}{h}\nabla\left(v^2\right)
=\left(\frac1{v^2}+k\right)2v\nabla v\] which, thanks to the
$C^1$-estimates is bounded by $C\abs{ \na v}\leq C\abs{A}\leq C
G^{1/2}$.  We can thus estimate 
\begin{align*}
  I=&\,-((1-\theta)+t\theta)U^4\tfrac1h\langle\nabla h, \nabla
  G\rangle\\ 
  \le&\,C((1-\theta)+t\theta)U^3 v G\abs{\la \nabla v,\na U\ra}
  +C\theta\lambda v^3U^3\abs{\langle\nabla v,\nabla U\rangle}\\
  &\,\quad+4\theta\lambda U^4\left(1+kv^2\right)|\nabla
  v|^2\\
  \leq&\, C((1-\theta)+t\theta)U^3 G^{3/2}+C\theta\lambda U^3
  \abs{A}+5\theta\lambda U^4|\nabla v|^2,
\end{align*}
where we used that $kv^2\leq \frac1{24}$ as well as that $v$ is
bounded in the last step. Using Young's inequality, we can absorb the
first two terms of this estimate into the first two terms of the right
hand side of \eqref{est:w2} and another additive constant
$C(\lambda)$, while the last term is absorbed into the third term of
\eqref{est:w2}.
  
  Furthermore, the terms appearing in 
   \begin{align*}
     II=&\,-8((1-\theta)+t\theta)U^3\langle\nabla U,\nabla
     G\rangle\umbruch\\
     =&\,32((1-\theta)+t\theta)U^2G|\nabla U|^2 +32\theta\lambda
     U^2v^2|\nabla U|^2 +16\theta\lambda U^3v\langle\nabla U,\nabla
     v\rangle\\
     \leq&\, C ((1-\theta)+t\theta)U^2G+C+C\theta \lambda U^3\abs{A}
  \end{align*}
  can also be absorbed into the first two terms on the right hand side
  of \eqref{est:w2} and a constant.
  
  Finally, to analyse $III_\eps$, we recall that $\la \na \aleps,
  e_{n+2}\ra=\beta_\epsilon' \big(\langle\nu_{\O}, e_{n+2}\rangle
  -\tfrac1v\langle\nu_{\O},\nu\rangle\big)\leq
  \frac{\abs{\bepsp}}{v}\la \nu_\Ob,\nu\ra$.  Thus 
  \begin{align*}
    III_\eps \leq&\,-\frac12 \la\nu,\nu_\Ob\ra_+\abs{\bepsp}\cdot
    \left[((1-\theta+t\theta)U^4G-4\theta\lambda U^4v^2\right]\\
    =&\,-\frac12 \la\nu,\nu_\Ob\ra_+\abs{\bepsp}\cdot
    \left[w-5\theta\lambda U^4v^2\right] 
  \end{align*}
  is negative in points where $w\geq D$ provided $D$ is chosen
  sufficiently large.

  All in all we thus conclude that we can fix a number $\lambda\geq 1$
  so that the estimate
  \[\left(\tdt-\Delta\right)w\le-\tfrac{c_1}4((1-\theta)+t\theta)U^4G^2
  +C\]
 holds in every point in which $\na w=0$ and $w\geq D$.
 
 We finally remark that in points where $w$ is large 
 the first term in this estimate dominates since also 
 $$((1-\theta)+t\theta)U^4G^2\ge((1-\theta)+t\theta)^2U^4G^2=U^{-4}
 (w-\theta \lambda U^4 v^2)^2$$ must be large as $U$ and $v$ are
 bounded above and as we only consider times $t\in[0,1]$ in case
 $\theta=1$.  Thus increasing $D$ further allows us to absorb the
 second term and yields the claim.
\end{proof}

 \subsection{$C^k$-estimates for solutions of the approximate problem}
 In order to guarantee the existence of solutions to the penalised
 flow \eqref{eps fluss} for all time, we show that, for each
 \textit{fixed} number $\eps>0$, solutions of \eqref{eps fluss}
 satisfy $C^k$-estimates for all positive times.
 
 We stress that these estimates are not uniform in $\eps$ and indeed
 that no such uniform control is possible as already the solutions of
 the stationary graphical obstacle problem are in general only in 
 $C^{1,1}$, see \cite{CGGlobalC11}. As such we shall refrain from
 writing down the explicit form of most terms and for the most part
 use the notation $B* C$ to denote arbitrary linear combinations of
 traces of $B\otimes C$ with respect to the metric.

 Recall that under the flow \eqref{eps fluss}, the metric evolves
 according to \eqref{eq:evol-g}, so that its Christoffel-symbols
 $\Gamma$ satisfy
 $$\tdt \Gamma= \na A* A+   \na \aleps* A+\aleps * \na A.$$

Since $\dt\na B=\na \dt B+\dt\Gamma * B$ for any
tensor $B$, we thus get
$$\tdt \na^m A=\na^m\tdt A+\na^{a_1} A * \na^{a_2} A*
\na^{a_3} A+\na^{a_1} \aleps * \na^{a_2} A* \na^{a_3} A,$$ 
where $a_i\in \N_0$ range over all triples with $a_1+a_2+a_3=m$.

The evolution equation for the second fundamental form for the general
flow \eqref{eq:general-flow} is known to be
\[\tdt A_{ij}=\nabla_i\nabla_jf-fA_i^kA_{kj}\]
which implies
\[\tdt A_{ij} -\Delta A_{ij} =|A|^2 A_{ij} -2HA_i^kA_{kj}
  +(H-f) A_i^kA_{kj} -\nabla_i\nabla_j(H-f).\]
For our flow we thus have
$$\left(\tdt-\Delta\right)A=-\na^2\aleps +\alpha_\epsilon*A*A +A* A*
A.$$ Using that $\Delta\na^mA=\na^m\Delta A+\na^{a_1} A * \na^{a_2} A*
\na^{a_3} A$, compare \eqref{Gauss formula}, we get
\begin{align*}
  \left(\tdt-\Delta\right) \abs{{\na^m A}}^2
  =&\, -2\abs{\na^{m+1}A}^2 +\na^{m+2}\aleps* \na^m A\\
  &\,+(\na^{a_1}A+\na^{a_1} \aleps )* \na^{a_2} A* \na^{a_3} A*\na^m A,
\end{align*}
$a_1+a_2+a_3=m$. Since $A$ is bounded on $M_t\cap \left\{x^{n+2}\geq
  \ell\right\}$ and since in such regions the depth of penetration is
controlled by the results of Section \ref{sect:penetration}, we
conclude that in this region
\begin{equation}
\label{est:higher-order}
\left(\tdt-\Delta\right)\abs{\na^m A}^2\leq -2\abs{\na^{m+1} A}^2+C\cdot
\left|\nabla^mA\right|^2+C+C\eps^{-2(m+2)},
\end{equation}
with $C$ depending on $\ell$, the $C^{m+3}$-norm of the obstacle
$\psi$, bounds on $A$, $\nabla A$, \ldots, $\nabla^{m-1}A$, and either
a lower bound on $t$ or a bound on the second fundamental form of the
initial surface.

\begin{remark}\label{rem:higher-order}
  For fixed $\epsilon>0$, we deduce iteratively estimates for
  $|\nabla^mA|$, $m=1,2,\ldots$ for solutions of \eqref{eps fluss} of
  the following form:
  \begin{enumerate}[(i)]
  \item for any $\ell\in\R$, any $0<\tau$ and any $m\in\N$,
    there is a constant $C$ depending on $\eps>0$, on the obstacle and on
    local $C^1$-bounds of $M_0\cap\left\{x^{n+2}>\ell-1\right\}$, so
    that $\left|\nabla^m A\right|\le C$ in
    $M_t\cap\left\{x^{n+2}\ge\ell\right\}$, $t\ge\tau$. 
  \item if $M_0\cap\left\{x^{n+2}>\ell-1\right\}$ is additionally in
    $C^{m+2}$, then these estimates are valid up to time $t=0$,
    i.\,e.{} $\left|\nabla^m A\right|\le C$ in
    $M_t\cap\left\{x^{n+2}\ge\ell\right\}$, $t\in[0,\infty)$.
  \end{enumerate}
\end{remark}
\begin{proof}
  We may proceed as in the proof in the situation without obstacles,
  see \cite[Theorem 5.9]{OSMarielMCFwithoutSing}, after replacing the
  set where $u<0$ with the one where $U>0$ due to the different
  orientation of the graphs.  If the derivatives
  $\left|\nabla^kA\right|$, $1\le k\le m-1$, are already uniformly
  bounded in the set considered, the evolution equations for
  $\left|\nabla^mA\right|^2$ and $\left|\nabla^{m-1}A\right|^2$ are of
  the same form as in the proof of \cite[Theorem
  5.9]{OSMarielMCFwithoutSing}. Note that the constants $c$ will now
  depend on $\frac1\epsilon$. This, however, does not cause problems
  as we do not claim that these estimates are independent of
  $\epsilon$. When we compute the evolution equation
  of \[tU^2\left|\nabla^mA\right|^2 +\lambda
  \left|\nabla^{m-1}A\right|^2,\] we get an additional term
  $2tu\alpha_\epsilon \langle\nu,e_{n+2}\rangle
  \left|\nabla^mA\right|^2$, which can easily be absorbed. The rest of
  the argument carries over to the present situation.
\end{proof}

\section{Existence of approximate solutions}
We construct smooth approximate solutions to \eqref{MCF} depending
on parameters 
\begin{itemize}
\item $\epsilon\in (0,1)$ controlling the penalisation,
\item $L\in\R$, the height at which we truncate our initial value,
\item $R>3$, the radius of the ball on which we solve a Dirichlet
  problem, and
\item $\delta\in (0,1)$ to mollify both the truncated initial values
  and the obstacle. 
\end{itemize}

Given an obstacle $\Ob$ with $\pO=\graph(\psi)$ for a
$C^{1,1}_{\loc}$-function $\psi$ as described in Definition \ref{def: ini
  data}, we extend $\psi$ by $-\infty$ to $\R^{n+1}$. Then we mollify
$\Ob$ and consider the obstacles $\Ob^\de$, $\delta\in(0,1]$,
characterised by $\partial \Ob^\delta=\graph(\psi_\delta)$,
where $$\psi_\delta=\psi* \eta_\delta$$ for a smooth mollification
kernel $\eta_\delta=\delta^{-(n+1)}\eta(\cdot/\delta)$,
$\supp\eta\subset B_1(0)$, and let $$\alpha_\epsilon^\delta=\beps\circ
\dist_{\partial \Ob^\delta}$$ be the corresponding penalisation
function.

We remark that all results derived in the previous sections (except
for the higher order estimates of Remark \ref{rem:higher-order}) are
valid with constants independent of $\delta$ for this whole family of
obstacles as $(\Ob^\delta)_{\delta\in (0,1]}$ satisfy uniform
$C^2_\loc$-estimates.

We remark that mollifying the initial value $u_0$ with the same kernel
ensures that $u_0\ge\psi$ remains true after mollification.

In order to apply the results derived in the previous sections, we
shall furthermore only consider parameters so that
\beq\label{ass:eps-R1} \eps\leq \eps_0(L)\text{ the constant of Lemma
  \ref{lemma:penetration}} \eeq and so that $R$ is large enough to
guarantee that the initial map $u_0$ satisfies \beq\label{ass:eps-R2}
u_0\leq L-1 \text{ outside } B_{R/2}(0). \eeq

We then have the following existence result for approximate solutions.
\begin{proposition}
  \label{prop:approx-sol-exist}
  Let $u_0$ and $\Ob$ with $\partial\O=\graph\psi|_{\P}$ be an initial
  map and an obstacle as described in Definition \ref{def: ini data}
  and let $\O^\delta$, $\delta\in(0,1]$, be the mollified obstacles as
  described above.
  
  Then for every quadruple $(\eps,\delta,L,R)\in (0,1)^2\times
  \R\times [3,\infty)$ of parameters for which the assumptions
  \eqref{ass:eps-R1} and \eqref{ass:eps-R2} are satisfied, there exists
  a smooth solution $u^{\delta,L}_{\epsilon,R}\colon
  B_R(0)\times[0,\infty)\to\R$ to \beq \label{eq:approx-sol}
  \begin{cases}
    \dot u=\sqrt{1+|Du|^2}\cdot\left(\divergenz
      \left(\fracd{Du}{\sqrt{1+|Du|^2}}\right)
      +\alpha_\epsilon^\delta\right)
    &\text{in }B_R(0)\times[0,\infty),\\
    u=L&\text{on }\partial B_R(0)\times[0,\infty),\\
    u(\cdot,0)=(\max\{u_0,L\})_\delta =\max\{u_0,L\}*\eta_\delta
    &\text{in }B_R(0).
  \end{cases}
  \eeq
  
  Furthermore, for any $\ell>L+2$, there exists a constant
  $C=C(u_0,\O,\ell)$ such that 
  \begin{equation}
    \label{est:approx-sol}
    \left|Du^{\delta,L}_{\epsilon,R}(x,t)\right|+\sqrt t\cdot\left|D^2
      u^{\delta,L}_{\epsilon,R}(x,t)\right| +\sqrt 
    t\cdot\left|\tdt u^{\delta,L}_{\epsilon,R}(x,t)\right|\le C
  \end{equation}
  in every $(x,t)\in B_R\times [0,\infty)$ with $u(x,t)\ge \ell$.
  
  Here, the function $\alpha_\epsilon^\delta$ is evaluated at the
  point $\big(x,u^{\delta,L}_{\epsilon,R}(x,t)\big)$ on the evolving
  surface $\graph u^{\delta,L}_{\epsilon,R}(\cdot,t)$.
\end{proposition}
\begin{proof}
  The choice of $R$ implies that the smooth initial map
  $(\max\{u_0,L\})_\delta$ is constant near $\partial B_R(0)$, so that
  compatibility conditions of any order are fulfilled for the initial
  value problem \eqref{eq:approx-sol}.  Standard parabolic theory
  hence gives the existence of a smooth solution
  $u\equiv u^{\delta,L}_{\epsilon,R}$ defined on a maximal time interval
  $[0,T)$, with $T>0$.
  
  To establish long time existence it is thus sufficient to show that
  the derivatives of $u$ remain bounded for all times which we shall
  prove using a combination of standard techniques for mean curvature
  flow as well as the evolution equations derived in the previous
  sections. We remark that in this part of the proof we do not claim
  that any of the derived bounds are independent of the choice of the
  parameters but will rather prove the uniform a priori bounds
  \eqref{est:approx-sol} separately later on.

  To begin with, we observe that since $\alpha_\epsilon^\delta\ge0$,
  any constant function is a subsolution of the equation; in
  particular the constant $L$ serves as a lower barrier for $u$.
  Furthermore, the constant $\max\{\sup\psi,\sup u_0,L\}$ is a
  solution to the flow equation as $\alpha^\delta_\epsilon$ vanishes
  on its graph, so it is an upper barrier and our solutions remains
  uniformly bounded for all times.
  
  We remark that due to our choice of $R$, we have $u(x,t)\ge L>
  \psi(x)$ for any $|x|\ge R/2$ and any $t>0$. Hence $u$ evolves
  according to graphical mean curvature flow in any annulus
  $(B_\rho\setminus B_\sigma(0))\times[0,T)$ with
  $R/2<\sigma<\rho<R$. Standard theory, see \cite{EckerHuiskenInvent},
  implies uniform estimates for arbitrary derivatives of $u$ away from
  $t=0$ in such annuli.

  A priori estimates near the boundary follow as in
  \cite{HuiskenBdry}: Comparison with minimal surfaces yields boundary
  gradient estimates. The evolution equation of $v$ then implies
  gradient estimates in the annulus $B_R\setminus
  B_{\frac{3R}4}(0)$. Finally, uniform parabolicity of the equation
  leads to bounds on arbitrary derivatives of $u$ in this annulus away
  from $t=0$.

  To derive estimates in the interior, say on $B_{\frac{3R}4}$, we can
  now apply the maximum principle on $B_{\frac{3R}4}$ to the various
  evolution equations derived in the previous sections since we have
  already obtained bounds on the annulus and thus in particular on
  $\partial B_{\frac{3R}4}$; namely gradient estimates now follow from
  Lemma \ref{v evol} and Remark \ref{rem:sign nabla alpha eta},
  estimates on the second fundamental form follow from Lemma
  \ref{lemma:mod-f} and higher order estimates follow from
  \eqref{est:higher-order}.  This concludes the proof of long time
  existence.

  We finally observe that Propositions \ref{prop:v-est} and
  \ref{prop:C2-est} give a priori estimates for the gradient function
  and the second fundamental form, and thus for both $Du$ and $D^2u$,
  of precisely the form claimed in \eqref{est:approx-sol}, in
  particular with a constant that is independent of any of the
  parameters used in the construction.

  These estimates then imply the claim on the time derivative made in
  \eqref{est:approx-sol} since $u$ solves equation
  \eqref{eq:approx-sol} and since the penetration depth, and thus
  $\alpha_\epsilon^\delta$, is a priori controlled according to
  Corollary \ref{cor:K-introd}.
\end{proof}

\section{Proofs of the main results}\label{sect:proofs}
We are now able to prove the existence of viscosity solutions of
graphical mean curvature flow with obstacles as claimed in Theorems
\ref{thm1} and \ref{thm2}. 

\begin{proof}[Proof of Theorem \ref{thm1}]
  Let $\Ob$ and $u_0$ be an obstacle and initial condition as in
  Definition \ref{def: ini data} and let
  $u^{i}=u^{\delta_i,L_i}_{\epsilon_i,R_i}$ be any sequence of
  approximate solutions as constructed in Proposition
  \ref{prop:approx-sol-exist} for which
  $(\epsilon_i,L_i,\delta_i,R_i)\to(0,-\infty,0,\infty)$.
 
  Then the uniform $C^{2;1}$ estimates stated in
  \eqref{est:approx-sol} allow us to apply the variant of the theorem
  of Arzel\`a-Ascoli from \cite[Lemma 7.3]{OSMarielMCFwithoutSing}: we
  obtain a subsequence $u^i$ converging to a limiting function $\tilde
  u:\R^{n+1}\times [0,\infty)\to \R\cup\{-\infty\}$ which induces a
  pair $(\Omega,u)$ consisting of
 $$\Omega:=\{(x, t):u(x,t)>-\infty\}\subset \R^{n+1}\times [0,\infty)$$
 and the restriction $u:=\tilde u\vert_{\Omega}\colon\Omega\to\R$.
 Here the convergence of $u^i\to \tilde u$ is pointwise everywhere and
 in $C_{loc}^{1,\alpha;0,\alpha}(\Omega\cap \{t>0\})$ for every
 $\alpha\in (0,1)$.

\phantom{x}

We recall furthermore, that the graphical velocity of the approximate
solutions is controlled by \eqref{est:approx-sol}. Therefore the
approximate solutions satisfy uniform parabolic H\"older estimates up
to time $t=0$ on any compact subsets of $\Omega$ and so the obtained
limit $u$ is in $C_{loc}^0(\Omega)$ and attains the desired initial
value $u(0)=u_0$.

\phantom{x}

We now prove that $u$ is a viscosity solution of
\begin{equation}
   \label{eq:viscobst}
 \min\left\{ \dot u-\sqrt{1+|Du|^2}\cdot
   \divergenz\left(\frac{Du}{\sqrt{1+|Du|^2}}\right),
     u-\psi\right\}=0.
 \end{equation}

We recall that $u$ is a viscosity subsolution for the
above operator if for any point $(x_0,t_0)$, the left-hand side of
\eqref{eq:viscobst} is nonpositive for all $C^2$-functions $\phi$
satisfying $\phi(x_0,t_0)=u(x_0,t_0)$ as well as $u(x,t)\le\phi(x,t)$
for all $(x,t)\in\Omega$ with $t<t_0$.


To begin with, we observe that the estimates on the penetration depth
derived in Section \ref{sect:penetration} imply that
$$u(x,t)\geq \psi(x)$$
for every $(x,t)\in \Om$. We can thus distinguish between points with
$u(x_0,t_0)> \psi(x_0)$ and points where the surface touches the
closure of the obstacle.

In the former case it is clearly enough to show that $u$ is locally a
viscosity solution of the graphical mean curvature flow equation
\beq \label{eq:graph-MCF} \dot u+\sqrt{1+|Du|^2}\cdot H=0.\eeq Given
such a point $(x_0,t_0)$ in which $u(x_0,t_0)>\psi(x_0)$, we observe
that in a space time neighbourhood also $u^i(x,t)\geq \psi(x)$ for $i$
sufficiently large, since these functions converge locally uniformly
to $u$.  Consequently the functions $u^i$ are classical solutions of
\eqref{eq:graph-MCF} in this neighbourhood.  As we have locally
uniform gradient estimates for the functions $u_i$, equation
\eqref{eq:graph-MCF} is uniformly parabolic, so arguing as in
\cite[Proposition 2.9]{CaffarelliCabre}, we obtain that the limit $u$
is indeed a viscosity solution to \eqref{eq:graph-MCF}.
\par
It remains to consider points $(x_0,t_0)$ with $
u(x_0,t_0)=\psi(x_0)$.  First of all, since the second argument in the
minimum in \eqref{eq:viscobst} is zero for every $C^2$ function $\phi$
with $\phi(x_0,t_0)=u(x_0,t_0)$, the condition that this minimum is
non-positive in the viscosity sense is clearly satisfied. It remains
to show that $\dot u+\sqrt{1+|Du|^2}\cdot H\geq 0$ holds in the
viscosity sense. But $\aleps\geq 0$, so the functions $u^i$ satisfy
this inequality classically on the whole domain of definition so that
passing to the limit as explained above implies that $u$ itself
satisfies the inequality in the viscosity sense. We conclude that $u$
is a viscosity solution to \eqref{eq:viscobst}.

\phantom{x}

The claimed estimate \eqref{est:thm-v-A-bound-lip-ini-data} follows
from \eqref{est:approx-sol}. For a $C^{1,1}_{\loc}$-initial
hypersurface we can furthermore derive bounds on the second
fundamental form up to $t=0$ from Lemma \ref{tU4g lambda U4v2 lem}.

\phantom{x}

Consider finally a point $(x,t)\in\Omega$ with $t>0$ that is not
contained in the contact set $\Gamma$, i.\,e.{} such that
$u(x,t)>\psi(x)$. By uniform convergence we also have $u^i>\psi$ in a
neighbourhood of $(x,t)$ for sufficiently large $i$. Thus $u^i$ evolves
by graphical mean curvature flow in this neighbourhood. As the $u^i$
satisfy locally uniform gradient estimates we may apply the interior
estimates of \cite[Theorems 3.1, 3.4]{EckerHuiskenInvent} and deduce
smoothness of $u$ in a smaller neighbourhood of $(x,t)$.
\end{proof}

\begin{proof}[Proof of Theorem \ref{thm2}]
  We proceed as in the proof of Theorem \ref{thm1} and consider
  approximate solutions $u_{\epsilon,R}^\delta$ as in Proposition
  \ref{prop:approx-sol-exist} but now with the initial and boundary values
  in \eqref{eq:approx-sol} replaced with
  $u(x,t)=u_0(x)\equiv\lim\limits_{|y|\to\infty}u_0(y)$ on $\partial
  B_R(0)\times[0,\infty)$ for large $R>0$ and $u(\cdot,0)=u_0*
  \eta_\delta$ in $B_R(0)$. Using large spheres near infinity as
  barriers, we can separate the evolving graph from the obstacle near
  infinity. Thus $u_{\epsilon,R}^\delta$ solves graphical mean
  curvature flow without additional terms due to the obstacle outside
  of a compact set that does not depend on $R$ but may grow in
  time. In this region, we can thus apply the a priori estimates of
  \cite{EckerHuiskenInvent} and obtain uniform bounds on arbitrary
  derivatives of $u_{\epsilon,R}^\delta$. \par As the additional term
  $\alpha_\epsilon^\delta$ is nonnegative, a hyperplane at height
  $\inf u_0-1$ acts as a lower barrier. Therefore $u_0$ can at most
  penetrate into a bounded subset of the obstacle and we can apply the
  maximum principle with $f_0$ equal to a constant in Lemma
  \ref{lemma:comp-surf-zwei-ff}. Then we obtain bounds on derivatives
  of $u_{\epsilon,R}^\delta$ by applying the maximum principle
  directly (i.\,e.{} without localising with $U-\ell$) to the
  evolution equations for $v$ of Lemma \ref{v evol}, for $G$ of Lemma
  \ref{lemma:mod-f} and to \eqref{est:higher-order} for higher order
  derivatives.  This is possible since far away from the origin those
  quantities are controlled by the estimates of
  \cite{EckerHuiskenInvent}, so that we can apply the maximum
  principle on compact sets. This implies spatial $C^2$-estimates that
  depend neither on $\epsilon,\,\delta$ nor $R$ and higher order
  estimates that depend only on $\epsilon$ but not on $\delta$ or
  $R$. \par
  Then arguing as in the proof of Proposition
  \ref{prop:approx-sol-exist} yields the analogue of this proposition,
  in particular estimate \eqref{est:approx-sol} on all of
  $B_R(0)\times[0,\infty)$. Thus the arguments of the proof of Theorem
  \ref{thm1} also apply to the present situation and yield the desired
  result.
\end{proof}

\section{Geometric interpretation: back to the original problem}
We finally discuss how the graphical solutions constructed in the
previous sections can lead to a notion of weak solutions for the
original problem of flowing a general (in particular not necessarily
graphical) hypersurface $N_0$ in $\R^{n+1}$ in the presence of an
obstacle $\Pro\subset \R^{n+1}$. We consider the case of a
\textit{one-sided} obstacle, intuitively speaking an obstacle such
that either all or none of its components are enclosed by the initial
hypersurface. This includes of course the special case of a connected
obstacle.

To be more precise, let $d_{N_0}$ be a continuous distance function to
$N_0$ which has non-vanishing gradient on $N_0$ (and thus changes sign
as we pass through $N_0$). We then ask that $d_{N_0}$ has constant
sign on all of $\Pro$, say $d_{N_0}\vert_{\P}<0$ and consider a
complete graphical initial hypersurfaces over $\Omega_0:=\{x\colon
d_{N_0}(x)<0\}$ and a complete graph over the obstacle as in
Definition \ref{def: ini data}. This construction requires no
regularity of the initial surface $N_0$ or the obstacle $\P$.

Let now $(\Omega,u)$ be the corresponding singularity resolving
solution whose existence for all times we have proven above. 

Let $\Omega_t$ be the time-slice of $\Omega$ at time $t$ as in
Definition \ref{def:gr-mcf-obst} \eqref{def sol domain}. Then
$M_t:=\graph(u(\cdot,t)\colon\Omega_t\to\R)$ is a complete
hypersurface and $(M_t)_{t\ge0}$ solves graphical mean curvature flow
respecting the obstacle, in particular, $u\ge\psi$. Thus $\Omega_t$
contains $\P$ and $\partial\Omega_t$ remains disjoint from the open
obstacle $\P$ for all times. Motivated by the results of M. S\'aez and
the second author, see in particular Proposition 9.2 of
\cite{OSMarielMCFwithoutSing} for further details, we can interpret
$(\Omega_t)_t$ as a weak solution to mean curvature flow with
obstacle.  The relation between this notion of a weak solution and the
level set formulation for mean curvature flow with obstacles, cf.{}
\cite{MercierObstacleViscosity}, will be analysed in future work.

\def\emph#1{\textit{#1}}
\bibliographystyle{amsplain} 

\end{document}